\documentclass{siamltex}
\usepackage[latin1]{inputenc}
\usepackage[english]{babel}
\usepackage{graphicx}
\usepackage{amsmath,amsfonts,amssymb,amsopn,amscd}
\usepackage{bbm,wasysym,stmaryrd}
\usepackage{subfigure}
\usepackage{color}
\usepackage{hyperref}
\usepackage[normalem]{ulem}
\graphicspath{{./Figs}}

\newcommand{\R}{\mathbb{R}}
\newcommand{\N}{\mathbb{N}}

\newcommand{\erf}{\mathrm{erf}}

\newcommand{\der}[2]{ \frac{\text{d} #1}{\text{d} #2} }  

\newenvironment{remark} {\par {\noindent \it \sc Remark.} \small \it } {}
\newenvironment{remarks} {\par {\noindent \it \sc Remarks.} \small \it } {}

\newcommand{\Exp}[1]{\mathbb{E}\left [ #1 \right]}

\newcommand{\eps}{\varepsilon}
 
\newcommand{\M}{\mathcal{M}}

\graphicspath{{Figures/}}

\title{Noise-induced canard and mixed-mode oscillations in large stochastic networks with multiple timescales}

\author{ Jonathan Touboul\footnotemark[3]\ \footnotemark[4]\ \footnotemark[1] \and Martin Krupa\footnotemark[2] \and Mathieu Desroches\footnotemark[2]}

\begin{document}
\maketitle
\renewcommand{\thefootnote}{\fnsymbol{footnote}}
\footnotetext[3]{The Mathematical Neurosciences Laboratory, Center for Interdisciplinary Research in Biology (CNRS UMR 7241, INSERM U1050, UPMC ED 158, MEMOLIFE PSL*) Coll\`ege de France, 11 place Marcelin Berthelot, 75005 Paris, }
\footnotetext[4]{BANG Laboratory, Inria Paris-Rocquencourt.}
\footnotetext[1]{\texttt{jonathan.touboul@college-de-france.fr}}
\footnotetext[2]{SISYPHE Laboratory, Inria Paris-Rocquencourt.}

\begin{abstract}
	We investigate the dynamics of large stochastic networks with different timescales and nonlinear mean-field interactions. After deriving the limit equations for a general class of network models, we apply our results to the celebrated Wilson-Cowan system with two populations with or without slow adaptation, paradigmatic example of nonlinear mean-field network. This system has the property that the dynamics of the mean of the solution exactly satisfies an ODE. This reduction allows to show that in the mean-field limit and in multiple populations with multiple timescales, noise induces canard explosions and Mixed-Mode Oscillations on the mean of the solution. This sheds new light on the qualitative effects of noise and sensitivity to precise noise values in large stochastic networks. We further investigate finite-sized networks and show that systematic differences with the mean-field limits arise in bistable regimes (where random switches between different attractors occur) or in mixed-mode oscillations, were the finite-size effects induce early jumps due to the sensitivity of the attractor. 
\end{abstract}
\begin{keywords} 
mean field equations, neural mass models, bifurcations, noise, dynamical systems
\end{keywords}

\begin{AMS}
34E15, 
60F99, 
34E17, 
34F05, 
34C15, 
34C23, 
60B10, 
34C25  
\end{AMS}

\pagestyle{myheadings}
\thispagestyle{plain}
\markboth{M. DESROCHES, M. KRUPA \& J. TOUBOUL}{Stochastic Networks with multiple timescales}

\renewcommand{\thefootnote}{\arabic{footnote}}

\section*{Introduction}
At functional scales, cortical networks are composed of a large number of cells subject to an intense noise. Such networks process and encode information collectively, and in the variety of information transmitted, synchronous activity and periodic behaviors are ubiquitous~\cite{kandel2000principles,buzsaki2009rhythms}, both in physiological (associated to memory and attention) and pathological (epilepsy) states. Oscillatory patterns observed can be extremely complex. Theoretical descriptions of such complex oscillations have been satisfactorily performed in deterministic models of single neurons, and are generally attributed to the presence of several very different time scales in the dynamics. The analysis of these systems shed light on important and sensitive dynamical scenarios such as canard explosions (see for example \cite{krupa-kopell:08}). The question we address in this work is whether such oscillatory patterns can arise in the context of large-scale neuronal aggregates, including a very large number of neurons, strongly interconnected and subject to noise. In other words, we will be interested in complex oscillatory patterns that can appear in large scale neuronal assemblies in the presence of noise. As was the case of deterministic small dimension model, we will specifically focus on the {\it canard phenomenon} and the related phenomenon of {\it Mixed Mode Oscillations} (MMOs). 

Canard phenomenon, known also as {\em canard explosion} is a transition from a small, Hopf type oscillation to a relaxation oscillation, occurring upon variation of a parameter. This transition has been first found in the context of  the van der Pol equation~\cite{eb-jlc-fd-md_81} and soon after in numerous models of phenomena occurring in engineering and in chemical reactions~\cite{mb_88}.
A common feature of all these models is the presence of time scale separation (slow and fast variables). 
A particular feature of canard explosion is that it takes place in a very small parameter interval. 
For the van der Pol system, where the ratio of the timescales is given by a small parameter $\eps$, the width of this parameter interval can be estimated by $O(\exp(-c/\eps)$, where $c>0$ is a fixed positive constant. The transition occurs through a sequence of {\it canard cycles} whose striking feature is that they contain segments following closely a stable slow manifold and subsequently an unstable slow manifold.  

The work on canard explosion led to investigations of slow/fast systems in three dimensions, with two slow and one fast variables. In the context of these systems the concept of a {\em canard solution} or simply {\em canard} has been introduced, as a solution passing from a stable to an unstable slow manifold \cite{eb_90,am-ps-hl-eg_98,ps-mw_01}. Canards arise near so called folded singularities of which the best known is folded node  \cite{eb_90,ps-mw_01,mw_05}. Unlike in systems with one slow variable canards occur robustly (in $O(1)$ parameter regions). Related to canards are Mixed Mode Oscillations (MMOs), which are trajectories that combine small oscillations and large oscillations of relaxation type, both recurring in an alternating manner. Recently there has been a lot of interest in MMOs that arise due to the presence of canards, starting with the work of Milik, Szmolyan, Loeffelmann and Groeller \cite{am-ps-hl-eg_98}.  The small oscillations arise during the passage of the trajectories near a folded node or a related folded singularity. The dynamics near the folded singularity is transient, yet recurrent: the trajectories return to the neighborhood of the folded singularity by way of a global return mechanism \cite{mb-mk-mw_06}.

Due to the sensitivity of this phenomena, one may expect that canard explosions disappear in the presence of noise. This is implied by the results of~\cite{berglund-gentz-kuehn:12}, who study a system with small noise. Using a pathwise approach, they show that small noise perturbations dramatically disturb the behavior of the slow-fast system they consider, hiding small oscillations in the random fluctuation and having a positive probability to escape from the canard solution. 
In this picture, it seems relatively hopeless to obtain at the levels of noise typical of the cortical activity, and at the dimension of the system, collective canard phenomena and individual canards with a high degree of reproducibility and regularity. However, in large scale stochastic systems, averaging effects can take place, regularizing the probability distribution of the macroscopic activity, as well as that of individual cells. These macroscopic behaviors and the averaging effects arising at the level of large networks are mathematically characterized by limit systems, the \emph{mean-field limits}. In the present manuscript  we show that canard explosion and complex oscillatory pattern appear in the presence of (not necessarily small) noise in these mean-field limits. Moreover, we show that such phenomena may arise through noise induced bifurcations. 

Usual computational neurosciences descriptions of large scale networks rely on heuristic models which were derived in the seminal works of Wilson, Cowan and Amari~\cite{amari:72,amari:77,wilson-cowan:72,wilson-cowan:73}. These models implicitly make the assumption that noise cancels out through averaging effects in the asymptotic regime where the number of neurons goes to infinity. The analysis of these heuristic models successfully accounted for a number of cortical phenomena such as spatio-temporal pattern formation in spatially extended models~\cite{ermentrout:98,ermentrout-cowan:79,bressloff-cowan-etal:02}. 

Indeed, in order to uncover, in large scale noisy system, the presence of fine effects such as complex oscillatory patterns or canard explosions, one needs to very precisely describe the limits of the network activity as the number of neurons increases. This is why we will apply a recent developed approach based on the theory of interacting particle systems to our networks. This approach was used in a very general context including nonlinear models of spiking neurons, spatial extension and propagation delays in~\cite{touboulNeuralfields:11}. The resulting macroscopic states are described through extremely complex equations, and in general do not provide sufficient information on the dynamics of the global variables. However, in a simple Amari-Hopfield type of models, it was shown in~\cite{touboulNeuralfields:11}
that under the condition that initial conditions are Gaussian that the equations can be exactly reduced to 
a set of deterministic differential equations  governing the two first moments of the solutions, and these were used to analyze the dynamics of the macroscopic regimes~\cite{touboul-hermann:11,touboulNeuralFieldsDynamics:11}. This approach is precisely the one chosen in the present article. 

In contrast with the previous studies in this domain, we will consider here that the neurons interact through a nonlinear function of a mean-field term corresponding to the current generated at the level of each population. This model, known as the Wilson-Cowan model, is biologically more relevant. Moreover, in the original article of Wilson and Cowan~\cite{wilson-cowan:72}, the system is introduced with a slow-fast structure that, to the best of our knowledge, was never analyzed in the perspective of canard explosions and MMOs. This model has been widely studied in the context of computational neuroscience, and seems particularly suited for the analysis of the dynamics of large networks. From the computational neuroscience viewpoint, this analysis will shed light on the finite-dimensional Wilson-Cowan system, since incidentally the reduced mean-field equations will be a particular case of Wilson-Cowan system. From the mathematical viewpoint, the mean-field equations arising in the case of nonlinear global coupling were never considered in literature and lead to a number of technical questions that will be treated in the present article. Moreover, that model has the advantage that the solutions of the mean-field equations have their statistical expectation satisfying a system of ordinary differential equations, regardless of the initial condition chosen, allowing for an application of bifurcation theory to analyze the system. 
The relative simplicity of the model will allow us to go very deep into the mathematical analysis of large scale neuronal networks in the presence of noise and different timescales, and address the problem of the presence of complex oscillatory patterns and canards in such systems.  

The paper is organized as follows. Section~\ref{sec:MFLimits} is devoted to the introduction of the models and the derivation of the mean-field equations. This section deals with general neuron models with nonlinear synapses, multiple populations, timescales and noise. We apply this analysis to the Wilson-Cowan system and derive the related moment equations, a set of ordinary differential equations. The obtained equations are then analyzed in section~\ref{sec:Bifurcations}, and we show that the system present canard explosions that can lead, in different contexts, to mixed-mode oscillations as the noise parameter is varied. These phenomena are recovered in numerical simulations of the finite-size network equations as shown in section~\ref{sec:Network} where finite-size effects are also analyzed.  

\section{Network models and their mean field limits}\label{sec:MFLimits}
This section introduces a general model of multi-population neuronal network and derives the related macroscopic mean-field equations arising as a limit when the network size goes to infinity.

\paragraph{Models}
We are interested in the large scale behavior arising from the nonlinear coupling of a large number $N$ of stochastic diffusion processes representing the membrane potential of neurons. Each neuron belongs to a specific population $\alpha\in\{1\cdots P\}$, characterizing the dynamics of neurons and their time constants, and the interaction with other neurons. Each population is composed of a large number of neurons denoted $(N_{\alpha},\alpha=1\cdots P)$. The state of neuron $i$ in population $\alpha$ is described by a variable $X^i_t\in E=\R^d$ typically describing the membrane potential of the cell and different ionic concentrations. The intrinsic dynamics of this variable (disregarding input and interactions with other neurons) is governed by a nonlinear function $f_{\alpha}$ and a time constant $\tau_{\alpha}$. 

The action of neurons of population $\beta$ onto neurons of population $\alpha$ will be modeled through a generalized version of Wilson-Cowan system, as a nonlinear sigmoidal transformation of a the average of the activity of all cells, weighted by the typical synaptic connectivity $J_{\alpha\beta}$. Beyond network interactions, neurons receive inputs corresponding to the sum of a deterministic current $I^{\alpha}_t$ only depending on the population they belong to, and stochastic current driven by independent stochastic process $\xi^i_t$ with common law (denoted with a slight abuse of notations $\xi^{\alpha}_t$), that will be considered to be Ornstein-Uhlenbeck processes, taking into account the synaptic filtering. 
In their most general form, the class of models analyzed here satisfy the network equations:
 \begin{equation}\label{eq:GeneralSigmoidInter}
	\tau_{\alpha}\der{X^i_t}{t} = f_{\alpha}(X^i_t) + S_{\alpha}\left ( X^i_t,\sum_{\beta=1}^P \frac{J_{\alpha\beta}}{N_{\beta}} \sum_{p(j)=\beta} X^j_t  + I^{\alpha}_t+\xi^i_t\right).
\end{equation}
where $S_{\alpha}(x,y)$ represents the effect of an input current $y$ onto a cell of population $\alpha$ with potential $x$, $\sum_{\beta=1}^P \frac{J_{\alpha\beta}}{N_{\beta}} \sum_{p(j)=\beta} X^j_t$ the total current received through the network, and $p:\N\mapsto \{1\cdots P\}$ is the population function associating to a neuron $j$ the population $\beta$ it belongs to. 

This class of models is particularly interesting since it allows dealing with a number of situations arising in neurosciences. A prominent example of this class that will be useful in our developments is the Wilson-Cowan model, or activity-based firing rate model~\cite{wilson-cowan:72,wilson-cowan:73}. This model describes the mean firing-rate variable of a neuron (a real variable) of neuron $i$ in the population $\alpha$ through the differential equation:
	\begin{equation}\label{eq:WCNet}
		\tau_{\alpha}\frac{dX^i_t}{dt} = -X^{i}_t + S_{\alpha} \Big(\sum_{\beta=1}^P \frac{J_{\alpha\beta}}{N_{\beta}} \sum_{p(j) = \beta} X^j_t + I^{\alpha}_t + \xi^i_t\Big)
	\end{equation}
In this model, the variable $X^i_t$ represents the activity of the cell, and the right hand side is the incoming firing-rate to the cell $i$, filtered through the voltage to rate function $S_{\alpha}$ characterizing the integration properties of the synapse when subject to a input current (see~\cite{ermentrout:98}).
	
Another model of this class allows taking into account finer phenomena arising in the cell's activity motivating the description of additional variables, such as slow adaptation currents~\cite{curtu-rubin:11} or slow modulation currents~\cite{izhikevich2000neural}. These models correspond to minimal descriptions of systems with slow adaptation proposed by Izhikevich in~\cite{izhikevich2000neural}: in two-dimensional systems, a variety of behavior can be accounted for using a single slow adaptation. We slightly generalize this model and define a multi-population networks with mean-field slow modulation current. These systems are described by the following equations, for neuron $i$ in population $\alpha$:
	\begin{equation}\label{eq:SlowModulation}
		\begin{cases}
			\tau_{\alpha}\der{X^i_t}{t}&=-X^i_t +S_{\alpha}(\sum_{\beta=1}^P \frac{J_{\alpha\beta}}{N_{\beta}}\sum_{p(j)=\beta} X^j_t +\lambda_{\alpha} U^i_t+I^{\alpha}_t+\xi^{i,1}_t)\\
			\der{U^i_t}{t}&= \varepsilon_{\alpha}(k+\gamma U^i_t-\sum_{\beta=1}^P \frac 1 {N_{\beta}} \sum_{p(j)=\beta} X^j_t)
		\end{cases}
	\end{equation}
	
	Beyond these simple models that will be exactly reducible to ODEs, the class of models includes classical models in neurosciences such as the Fitzhugh-Nagumo model ($X\in\R^2$, $f_{\alpha}$ has a cubic and an affine component, and $S(x,y)$ is either a linear function when considering electrical synapses or a sigmoid transform when considering chemical synapses) or the Hodgkin-Huxley system. These models differ from usual systems analyzed in the mean-field limit in that there is a nonlinear summation of the interactions, here through our voltage-to-rate function. Limits of such interacting stochastic processes were the number of neurons tends to infinity were never addressed for such interactions considering the nonlinear transform $S$. We now derive the limits of these stochastic networks equations as the number of neurons go to infinity.

\paragraph{Macroscopic Limits}

Based on the general network given by equations~\eqref{eq:GeneralSigmoidInter}, we derive the macroscopic mean-field limit under a few assumptions on the dynamics. For the sake of simplicity, we denote by $f(x): E^P\mapsto E^P$ (resp. $S(x,y):E^P\times E^P\mapsto E^P$) the function with component $\alpha$ equal to $f_{\alpha}(x^{\alpha})$ (resp. $S_{\alpha}(x^{\alpha}, y^{\alpha})$). We make the following assumptions:
\renewcommand{\theenumi}{(H\arabic{enumi})}
\begin{enumerate}
	\item $f$ is $K$-Lipschitz continuous
	\item $f$ satisfies the linear growth condition: $f(x) \leq C (1+\vert x \vert^2)$. 
	\item $S$ is $L$-Lipschitz continuous in both variables
	\item $S$ is bounded with supremum denoted $\Vert S \Vert_{\infty}$. 
\end{enumerate}
\begin{remarks}
	\begin{itemize}
		\item Rigorously, in the Fitzhugh-Nagumo or Hodgkin-Huxley system, $f$ is not globally Lipschitz continuous and does not satisfies a linear growth condition. However, these functions are locally Lipschitz continuous, and satisfy the estimate $x^T f(x)\leq K(1+\vert x\vert^2)$. These two assumptions allow to demonstrate the same kind of properties, by using It\^o formula and truncation method. This was done in~\cite{touboulNeuralfields:11} in a more general context, and the interested reader may readily apply this generalization to our system. 
		\item In these assumptions, we consider without loss of generality the same constants $K$ for the sake of simplicity. 
		\item All these assumptions are obviously satisfied for the Wilson-Cowan system. 
	\end{itemize}
\end{remarks}

We will show that in the mean-field limit, the propagation of chaos applies and that the state of all neurons in that limit satisfy the implicit equation:
\begin{equation}\label{eq:MFESigmoid}
	\der{X^{\alpha}_t}{t}=f_{\alpha}(X^{\alpha}_t) + S_{\alpha}\left ( X^{\alpha}_t,\sum_{\beta=1}^P J_{\alpha\beta} \Exp{X^{\beta}_t}  + I^{\alpha}_t+\xi^{\alpha}_t\right).
\end{equation}
which can be written in vector form:
\begin{equation}\label{eq:MFESigmoidVector}
	\der{X_t}{t}=f(X_t) + S\left ( X_t,J\cdot \Exp{X_t}  + I_t+\xi_t\right)
\end{equation}
where $J=(J_{\alpha\beta})_{\alpha,\beta\in\{1\cdots P\}}\in \R^{P\times P}$ and $I_t=(I^{\alpha}_t)_{\alpha=1\cdots P}$, $\xi_t=(\xi^{\alpha}_t,\alpha\in\{1\cdots P\})$.  

This equation is not a usual stochastic differential equation in that in involves the expectation of the solution. It is rather an implicit equation on the space of probability distributions. The first question that may arise in this type of unusual stochastic equation is its the well-posedness. This is subject to the following:
\begin{proposition}\label{pro:ExistenceUniqueness}
	There exists a unique, square integrable, strong solution to the equations~\eqref{eq:MFESigmoid} starting from a square integrable initial condition.
\end{proposition}
\begin{proof}
	The existence and uniqueness of solutions is demonstrated using a routine contraction argument. Let us fix a square integrable initial condition random variable $X_0\in E^{P}$ and define the map $\Phi$ acting on the space of square integrable stochastic processes $\M^2([0,T],E^{ P})$:
	\[\Phi: X \mapsto \bigg(X_0+\int_0^t \Big(f(X_s) + S\big ( X_s, J\cdot \Exp{X_s}  + I_s+\xi_s \big)\Big)\,ds \bigg)_{t \in [0,T]}.\]
	Fixed points of $\Phi$ are exactly the solutions of the mean-field equations. Moreover, for $X$ a square integrable, we have $\Phi(X)\in \M^2([0,T],E^P)$. 
	Indeed, thanks to Cauchy-Schwarz inequality, we have:
	\begin{align*}
		\Exp{\sup_{t\in[0,T]} \vert \Phi(X)_t\vert^2} &\leq 4 \bigg\{ \Exp{\vert X_0\vert^2}+ T \int_0^T C^2 (1+\Exp{\vert X_s\vert^2}) + \Vert S \Vert_{\infty}^2 \,ds \bigg\}<\infty.
	\end{align*}
	
	Fixing a square integrable process $ X^0 \in \M^2([0,T],E^{P})$, we define the sequence of square integrable processes $X^{n+1}=\Phi(X^n)$. We note $D^n_t=\Exp{\sup_{t\in[0,T]} \vert X^{n+1}_t-X^n_t\vert^2}$, $\Vert{J}\Vert$ the operator norm of $J$ and $Z^k_t=J\cdot \Exp{X_s}  + I_s+\xi_s$. We have:
	\begin{align*}
		D^{n}_t & \leq 3\,T \;\mathbbm{E}\Bigg\{\sup_{s\in [0,t]} \int_0^s \vert f(X^{n}_u)-f(X^{n-1}_u)\vert^2 + \vert S(X^{n}_u,Z^{n}_u)-S(X^{n-1}_u,Z^n_u)\vert^2\\
		& \quad \qquad  +\vert S(X^{n-1}_u,Z^{n}_u)-S(X^{n-1}_u,Z^{n-1}_u)\vert^2\,du \Bigg\}\\
		&\leq 3\,T \;\mathbbm{E}\Bigg\{\sup_{s\in [0,t]} \int_0^s (K^2+L^2) \vert X^{n}_u- X^{{n-1}}_u\vert^2 + L^2 \vert J\cdot \Exp{X^{{n}}_u - X^{{n-1}}_u}\vert^2 \,du \Bigg\}\\
		&\leq 3\,T \;\mathbbm{E}\Bigg\{\sup_{s\in [0,t]} \int_0^s (K^2+L^2) \vert X^{n}_u- X^{{n-1}}_u\vert^2 + L^2 \Vert{J}\Vert \;\Exp{\vert X^{{n}}_u - X^{{n-1}}_u\vert^2} \,du \Bigg\}\\
		&\leq 3\,T (K^2+L^2 (1+ \Vert{J}\Vert))\int_0^t D^{n-1}_s\,ds
	\end{align*}
	We hence obtain by immediate recursion:
	\[D^n_t \leq \frac{(K')^n T^n}{n!} D_T^0\]
	with $ K'=3\,T (K^2+L^2 (1+ \Vert{J}\Vert))$. $D_T^0$ is a finite quantity because of the boundedness of $\Exp{\vert \Phi(X^0)\vert^2}$. From this inequality, routine methods using the Bienaym\'e-Chebyshev inequality allow concluding on the existence and uniqueness of solutions of the mean-field equation (see e.g.~\cite{karatzas-shreve:87}). 
\end{proof}

The next theorem proves that the unique solution to the mean-field equations will describe exactly the behavior of the network when the network size goes to infinity (i.e. when all $N_{\alpha}\to\infty$), and that all neurons will be independent provided that their initial condition is. This latter property is referred as the \emph{propagation of chaos property}. 

\begin{theorem}\label{thm:MeanFieldSigmoid}
If the initial conditions to the network equations~\eqref{eq:GeneralSigmoidInter} are independent, and identically distributed population per population, then in the mean-field limit, the propagation of chaos applies and the laws of the network solutions converge towards that of the mean-field equations~\eqref{eq:MFESigmoid}, in the sense that for any $i\in\N$ there exists $\bar{X}^i_t$ with law given by the solution of the mean-field equations and such that:
\[\Exp{\sup_{t\in [0,T]}\vert X^i_t-\bar{X}^i_t \vert} \leq \frac{C(T)}{\min_{\alpha}\sqrt{N_{\alpha}}}\]
\end{theorem}	
\begin{proof}
The proof of the theorem uses the usual coupling argument popularized by Sznitman~\cite{sznitman:89} consisting in identifying an almost sure limit for the network equations. In details, considering neuron $i$ of population $\alpha$, we define the coupled process $\bar{X}^i_t$ solution of equation~\eqref{eq:MFESigmoid} with initial condition $X^i_0$ identical to the initial condition of neuron $i$ in the finite network and driven by the same noise $\xi^i_t$:
	\[\begin{cases}
		\der{\bar{X}^i_t}{t} &= f_{\alpha}(\bar{X}^i_t) + S_{\alpha}\left ( \bar{X}^i_t,\sum_{\beta=1}^P J_{\alpha\beta} \Exp{\bar{X}^{\beta}_t}  + I^{\alpha}_t+\xi^i_t\right)\\
		\bar{X}^i_0 &= X_0^i
	\end{cases}\]
	We denote by $\tilde{J}$ the maximal value of the coupling strengths $J_{\alpha\beta}$. We aim at showing that $X^{i}$ almost surely converges towards $\bar{X}^i$, and to this purpose analyze the distance between the two processes, $D^{\alpha}_t=\Exp{\sup_{t\in [0,T]} \vert X^{i,N}_t-\bar{X}^i_t\vert}$. Because of interchangeability in the network, that quantity does not depend on the specific neuron $i$ we consider but only on the neuronal population $\alpha$ it belongs to. For the sake of simplicity, we note:
	\[\begin{cases}
		\bar{Z}^{\alpha}_t &=\sum_{\beta=1}^P J_{\alpha\beta} \Exp{\bar{X}^{\beta}_t}  + I^{\alpha}_t+\xi^i_t\\
		\bar{Z}^{N,\alpha}_t &=\sum_{\beta=1}^P \frac{J_{\alpha\beta}}{N_{\beta}} \sum_{p(j)=\beta} \bar{X}^{j}_t  + I^{\alpha}_t+\xi^i_t\\
		Z^{N,\alpha}_t &=\sum_{\beta=1}^P \frac{J_{\alpha\beta}}{N_{\beta}} \sum_{p(j)=\beta} \Exp{X^{j}_t}  + I^{\alpha}_t+\xi^i_t.
	\end{cases}
	\]
	 In order to analyze this distance, we decompose it into the sum of four simpler terms that are more easily controllable:
	\begin{align*}
		\vert X^i_t-\bar{X}^i_t\vert &\leq \int_0^t \vert f_{\alpha}(X^i_s)-f_{\alpha}(\bar{X}^i_s)\vert + \vert S(X^i_s,Z^{\alpha,N}_s)-S(\bar{X}^i_s,Z^{\alpha,N}_s)\vert \\
		& \quad + \vert S(\bar{X}^i_s,Z^{\alpha,N}_s)-S(\bar{X}^i_s,\bar{Z}^{\alpha,N}_s)\vert 
		+\vert S(X^i_s,\bar{Z}^{\alpha,N}_s)-S(\bar{X}^i_s,\bar{Z}^{\alpha}_s)\vert \; ds\\
		& \leq (K+L)\int_0^t \vert X^i_s-\bar{X}^i_s\vert + L\tilde{J} \sum_{\beta=1}^P \frac 1 {N_{\beta}} \left(\sum_{p(j)=\beta} \vert X^j_t-\bar{X}^j_t\vert +  \vert \sum_{p(j)=\beta} \bar{X}^j_t - \Exp{\bar{X}^{j}_t}\vert\right)\,ds
	\end{align*}
	readily yielding:
	\begin{equation*}
		D^{\alpha}_t \leq \int_0^t (K+L) D^{\alpha}_s + L \tilde{J} \sum_{\beta=1}^P D^{\beta}_s +  L\tilde{J} \sum_{\beta=1}^P    \frac 1 {N_{\beta}} \Exp{ \left \vert \sum_{p(j)=\beta} \bar{X}^j_s - \Exp{\bar{X}^j_s} \right\vert }  \; ds
	\end{equation*}
	and hence for $D_t=\sum_{\alpha=1}^P D^{\alpha}_t$ we obtain:
	\begin{equation}\label{eq:DtBound}
		D_t \leq \int_0^t (K+L+ P\,L \tilde{J}) D_s
		 + P\,L \tilde{J}\; \sum_{\beta=1}^P  \frac 1 {N_{\beta}} \Exp{ \left \vert \sum_{p(j)=\beta} \bar{X}^j_u - \Exp{\bar{X}^j_u}  \right\vert} \; ds
	\end{equation}	
	Let us now evaluate the second term of the righthand side. We have, using Cauchy-Schwarz inequality:
	\begin{align*}
		\Exp{\left \vert \sum_{j\in\beta} \bar{X}^j_s - \Exp{\bar{X}^j_s} \right\vert} & \leq \Exp{\left\vert \sum_{j\in\beta} \bar{X}^j_u - \Exp{\bar{X}^j_u} \right\vert^2}^{1/2}
	\end{align*}
	The processes $(\bar{X}^j_s-\Exp{\bar{X}^j_s})$ are centered and independent since they are driven by independent processes $W^j_t$ and $\xi^j_t$. Developing the square of the sum, we are hence left with less that $N_{\beta}$ non-zero terms which are equal the variance of $\bar{X}^{\beta}_t$, namely $\Exp {(\bar{X}^j_s-\Exp{\bar{X}^j_s})^2}$ which is a bounded quantity as a result of proposition~\ref{pro:ExistenceUniqueness}. Let us denote by $C$ this bound.
	Returning to equation~\eqref{eq:DtBound}, we obtain:
	\[D_t \leq \int_0^t (K+L + P\,L \tilde{J}) D_s\,ds
	 + T\,P^2\,L \tilde{J} \sqrt{\frac{C}{\min_{\beta}N_{\beta}}}\]
	and Gronwall's lemma implies that:
	\[D_t \leq \frac{K_1}{\min_{\beta}\sqrt{N_{\beta}}} e^{K_2\,t}\]
	with $K_2=(K+L + P\,L \tilde{J})$ and $K_1=P\,L \tilde{J} \sqrt{C} T$. This quantity tends to zero when all $N_{\alpha}$ tend to zero (i.e. when $\min_{\beta} N_{\beta}\to\infty$), which implies the almost sure convergence $(X^i_t)$ towards $(\bar{X}^i_t)$, and hence convergence in law of $(X^i_t)$ towards the solution of the mean-field equations. Moreover, considering a finite number $l$ of neurons in the network $(i_1,\cdots,i_l)$, it is easy to see using the above result that $(X^{i_1}_t,\cdots, X^{i_l}_t)$ converge almost surely towards $(\bar{X}^{i_1}_t,\cdots, \bar{X}^{i_l}_t)$ which is a set of independent processes. We therefore have propagation of chaos, i.e. that in the limit where all $N_{\beta}$ tend to infinity, the state of these neurons are independent for all times. 
\end{proof}

\begin{remark}
	Note that this proof has been performed in the absence of Brownian additive noise in contrast with usual approaches. Yet, the noisy input received by each cell allows to recover the same kind of properties. We also emphasize the fact that the proof provided here readily extends to networks with additive noise. 
\end{remark}

As announced, the results of proposition~\ref{pro:ExistenceUniqueness} and theorem~\ref{thm:MeanFieldSigmoid} readily apply to the Wilson-Cowan model, as we now show: 
\begin{corollary}\label{cor-finred}
	In the mean-field limit where all population sizes $N_{\alpha}$ tend to infinity, the Wilson-Cowan system~\eqref{eq:WCNet} in the mean-field limit satisfies the propagation of chaos property and the law of neurons of population $\alpha$ asymptotically satisfy the mean-field equations:
	\begin{equation}\label{eq:WCMFE}
		\tau_{\alpha}\frac{dX^{\alpha}_t}{dt} = -X^{\alpha}_t + S \Big(\sum_{\beta=1}^P J_{\alpha\beta} \mu_{\beta}(t) + I^{\alpha}_t + \xi^{\alpha}_t\Big)
	\end{equation}
	where we denoted $\mu_{\beta}(t)=\Exp{X^{\beta}_t}$. These quantities $(\mu_{\beta}(t))_{\beta=1\cdots P}$ satisfy a closed equation:
	\begin{equation}\label{eq:meangeneral}
		\tau_{\alpha}\dot{\mu_{\alpha}} = -\mu_{\alpha}(t) + G_{\alpha}\Big(\sum_{\beta=1}^P J_{\alpha\beta} \mu_{\beta}(t) + I^{\alpha}_t\Big)
	\end{equation}
	where 
	\begin{equation}\label{eq-defG}
		G_{\alpha}:=x\mapsto \Exp{S_{\alpha}(x+\xi^{\alpha}_t)}.
	\end{equation}
	In the case of slow modulation currents~\eqref{eq:SlowModulation}, the mean-field equations read:
	\begin{equation}\label{eq:SlowModulationMFE}
		\begin{cases}
			\tau_{\alpha}\der{X^{\alpha}_t}{t}&=-X^{\alpha}_t +S(\sum_{\beta=1}^P J_{\alpha\beta} \mu_{\alpha}(t) +\lambda_{\alpha} U_t+I^{\alpha}_t+\xi^{\alpha}_t)\\
			\der{U_t}{t}&= \varepsilon_{\alpha}(k+\gamma U_t-\sum_{\beta=1}^P\mu_{\beta}(t))
		\end{cases}
	\end{equation}
	and in that case, $U_t$ is a deterministic variable satisfying an ODE and again, the expectation $\mu_{\alpha}(t)$ of the process $X^{\alpha}_t$ satisfies a closed set of ordinary differential equations:
	\begin{equation}\label{eq:SlowModulationMFEMoments}
		\begin{cases}
			\tau_{\alpha}\der{\mu_{\alpha}(t)}{t}&=-\mu_{\alpha}(t) +G_{\alpha}(\sum_{\beta=1}^P \mu_{\beta}(t)+ \lambda_{\alpha} U_t+I^{\alpha}_t)\\
			\der{U_t}{t}&= \varepsilon_{\alpha}(k+\gamma U_t-\sum_{\beta=1}^P\mu_{\beta}(t)).
		\end{cases}
	\end{equation}
\end{corollary}

\begin{proof}
Both cases are can be readily set in the general framework of the current section, hence proposition~\ref{pro:ExistenceUniqueness} and theorem~\ref{thm:MeanFieldSigmoid} apply, ensuring propagation of chaos and yielding the announced mean field limit. The averaged equation is trivial to derive taking the expectation of the process.
\end{proof}

We therefore proved that the systems under consideration here enjoy a considerable dimension reduction: the original system is a very large network of nonlinearly interacting stochastic processes with random currents, in which the mean is not solution of an ordinary differential equation, in contrast to the mean-field equations describing the infinite-size limit which gives rise to a closed-form ODE for the mean. 

The global dynamics of the system can hence be very precisely understood by analyzing the dynamics of the obtained system of differential equations. These involve an unspecified function, $G_{\alpha}$, which is the expectation of our sigmoidal transform $S_{\alpha}(x+\xi^{\alpha}_t)$ with respect to $\xi^{\alpha}$. In detail, if $\xi^{\alpha}_t$ has a density $p^{\alpha}_t$, this function writes:
\[G_{\alpha}(x)=\int_{\R} S(x+y)dp^{\alpha}_t(y),\]
coupling the dynamics of the mean equation to the distribution of the noisy current received by each cell. In what follows, we will consider that the noisy current is a stationary Ornstein-Uhlenbeck process, i.e. are centered, Gaussian processes with constant standard deviation denoted $\sigma_{\alpha}$. In the case where $S_{\alpha}(x)=\erf(g_{\alpha}\,x)$, a simple change of variables and integration by parts (see e.g.~\cite{touboul-hermann:11}) gives the function $G_{\alpha}$ in closed form: 
\begin{equation}\label{eq:EffectiveSigmoid}
	G_{\alpha}(x)=\erf(\frac{g_{\alpha}x}{\sqrt{1+g_{\alpha}^2\,\sigma_{\alpha}^2}}).
\end{equation}
The noise level $\sigma_{\alpha}$ is hence a parameter of the vector field, and may lead to \emph{noise-driven bifurcations}. 

From these fully determined systems, we are now in a position to use classical tools from the analysis of slow-fast systems to uncover the presence of canards and mixed-mode oscillations in macroscopic activity in relationship to the noise level $\sigma_{\alpha}$.

\section{Canards and MMO in the mean-field equations}\label{sec:Bifurcations}
We now analyze the limit systems~\eqref{eq:meangeneral} and~\eqref{eq:SlowModulationMFEMoments} with a particular focus on the effects of noise on the dynamics. We start by showing the role of noise in the emergence of canard explosion, prior to extend the model to slow modulation currents systems where the additional variable gives rise to noise-induced MMOs. 
\subsection{A two-populations Wilson-Cowan system: noise-induced canard explosion}\label{wc2d}
In this section, we study the mean-field equations~\eqref{eq:meangeneral} arising from Wilson-Cowan system with two populations. The network equations are given by~\eqref{eq:WCNet}, with $P=2$ populations evolving on vastly different timescales $\tau_1=\eps$ and $\tau_2=1$, $I^1=ze$, $I^2=0$. Moreover, in order to reduce the system to closed-form mean equations, we choose $S_{\alpha}$ to be $\erf$ functions and $\xi^{\alpha}$ to be a stationary Ornstein-Uhlenbeck process with standard deviation $\sigma_{\alpha}$. As a result of the previous section, we know that the associated mean-field equations~\eqref{eq:meangeneral} have solutions whose mean satisfy a closed set of ODEs given by the equations:
\begin{equation}\label{eq:WC2d}
\begin{cases}
	\eps\dot{u_1} &= -u_1 + G_1(J_{11}u_1+J_{12}u_2+z_e)\\
	   ~\dot{u_2} &= -u_2 + G_2(J_{21}u_1+J_{22}u_2)
\end{cases}
\end{equation}
In these equations, the two different effective nonlinearities $G_1$ and $G_2$ are of the form~\eqref{eq:EffectiveSigmoid} and therefore depend on the noise intensities $\sigma_{\alpha}$.

Let us now analyze the dynamics of these equations. We start by analyzing the behavior of the solutions upon variation of the parameter $z_e$, corresponding to fixed applied current to the first population. The behavior of the system as a function of current levels $z_e$ is particularly important in our context for its relationship with the slow modulation currents as will be made explicit in the following section. Moreover, this parameter plays the important role of moving the system from the excitable to the spiking regime. It is well known that such a parameter transition, under some degeneracy assumptions, can lead to a canard explosion, i.e. a transition from small amplitude oscillations to relaxation oscillations, see for example~\cite{mk-ps_01b}. And in the present system, canard explosions can be found by varying the parameter $z_e$ are observed (see bifurcation diagram generated with AUTO in Fig.~\ref{wc2dze}). The diagrams suggest that  the shown canard explosions are non-degenerate (see~\cite{mk-ps_01b} for a list of non-degeneracy conditions). It would be straightforward but tedious to verify these conditions and we omit such computations in this paper. This explosion informs us of the presence of canard explosions, upon variation of the input, in the Wilson-Cowan system~\eqref{eq:WC2d}, in the presence of noise. 

\begin{figure}[!t]
	\centering
        \subfigure[Bifurcation Diagram in $z_e$]{\includegraphics[width=.7\textwidth]{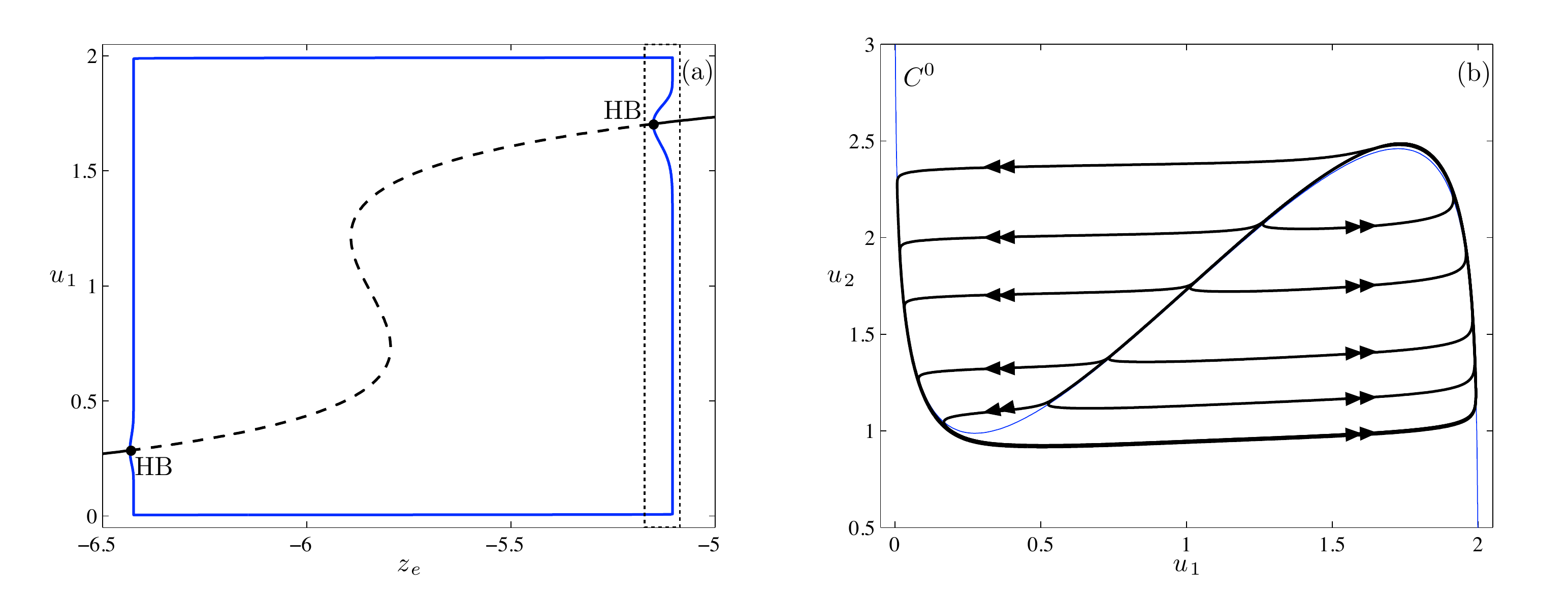}\label{wc2dze}}\\
        \subfigure[Bifurcation Diagram in $\sigma_1$]{\includegraphics[width=.7\textwidth]{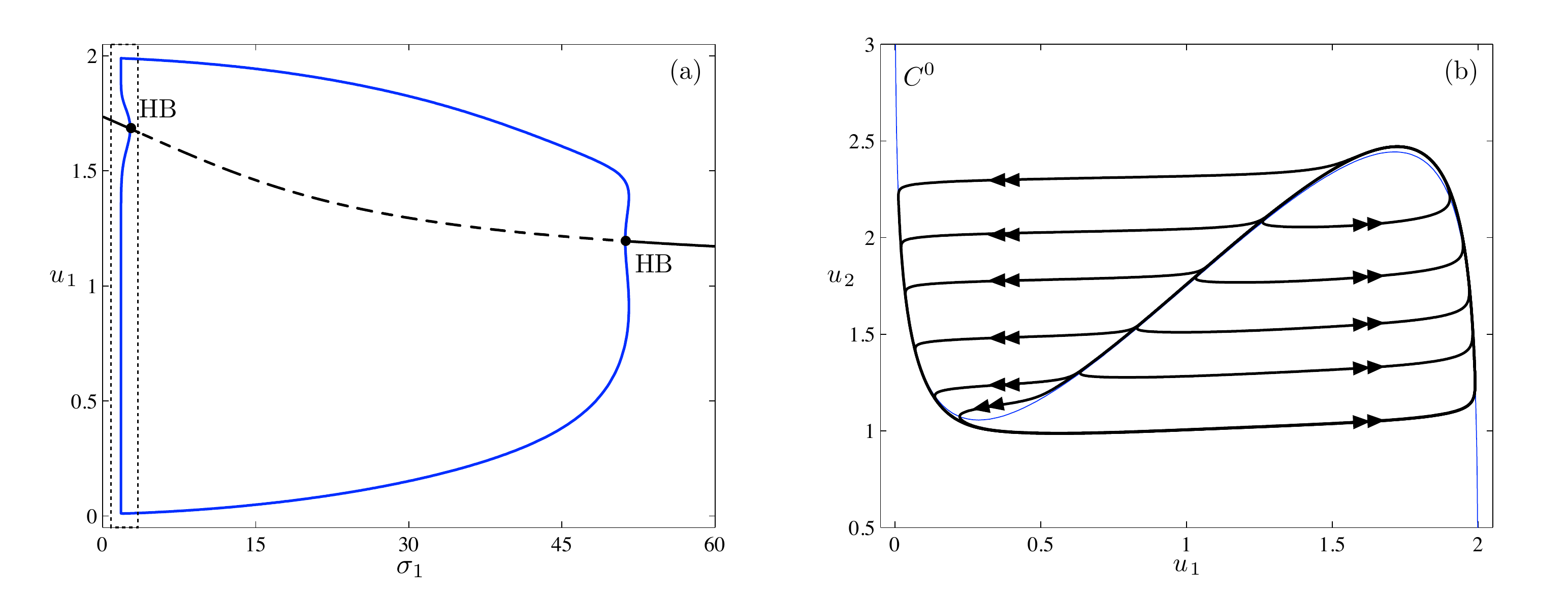}\label{wc2dsigma}}\\
	\caption{Canard explosion in the 2D WC model, for a variation of parameter $z_e$ (top panels) and $v_1$ (bottom panels). Shown are the bifurcation branch with the two explosions (left panels), cycles corresponding to the canard explosion highlighted by a dashed box on the bifurcation diagram (right panels).}
	\label{fig:wc2d}
\end{figure}

The question that arises is how this explosion depends on the amplitude of the noise neurons are submitted to. This can be addressed here using usual bifurcation and singular perturbation methods since the noise appears as a parameter of the ODEs, acting on the maximal slope of the sigmoids. Numerical simulation of the bifurcation diagram shows indeed that the presence of the canard explosion depends on the level of noise, and in particular that noise can induce canard explosions. As an example of such {\em noise induced
canard explosion}, we provide in Fig.~\ref{wc2dsigma} the AUTO bifurcation diagram of equations~\eqref{eq:WC2d} upon variation of $\sigma_1$ the noise received by neurons in the fast population. Therefore, noise can also bring the system from the region of excitability to the region of spiking, through  a non-degenerate canard explosion. 

As a side result, this analysis confirms the direct relationship between levels of noise and synchronized oscillations in the system~\cite{touboul-hermann:11,touboulNeuralFieldsDynamics:11}. This is an extremely surprising phenomena for which no microscopic interpretation has been provided thus far, and which in the context of slow-fast Wilson-Cowan system under consideration, appears even more surprising: indeed, the presence of this explosion shows that very small changes in the value of the variance of the noise yield dramatic qualitative switches on the macroscopic dynamics. Moreover, since we have proved that all neurons have the same probability distribution population-wise, the finding shows that the mean presents a sudden switch to a large-amplitude periodic orbit, corresponding to all neurons presenting an extreme change upon very small variation of the noise. More than noise-induced oscillations, the phenomenon shows how sensitive the dependence upon noise levels is in large-scale systems. 

Now that the two-population system has been analyzed as a function of parameters $z_e$ and $\sigma_1$, we proceed by adding to the same system a slow adaptation current. 

\subsection{The slow modulation current Wilson-Cowan system}
\label{wc3d}
We now enrich the network model analyzed in the previous section by considering slow adaptation currents. The underlying microscopic equation correspond to a two-populations Wilson-Cowan system with slow adaptation as introduced in equation~\eqref{eq:SlowModulation}. In order to simplify the model, we fix $\lambda_1=1$, $\lambda_2=0$ and $I^1=I^2=0$. The results of section~\ref{sec:MFLimits} ensure that the propagation of chaos occur and that the network converges towards mean-field equations, given by equation~\eqref{eq:SlowModulationMFE}, and that the mean of the solution to the mean-field equation satisfies an equation of type~\eqref{eq:SlowModulationMFEMoments}, which in our two populations setting with one fast and one slow population and parameters provided in section~\ref{wc2d}, read:
\begin{equation}\label{WC3d}
	\begin{cases}
		\eps\dot{u_1} &= -u_1 + G_1(J_{11}u_1+J_{12}u_2+z_e)\\
		   ~\dot{u_2} &= -u_2 + G_2(J_{21}u_1+J_{22}u_2)\\
		   ~\dot{z_e} &= k +\gamma z_e - u_1 - u_2,
	\end{cases}
\end{equation}
where the slow adaptation current (noted $U$ in the general equation~\eqref{eq:SlowModulationMFEMoments}) is denoted by $z_e$. This notation clearly shows that equation~\eqref{WC3d} can also be seen as generalization of the two population Wilson-Cowan system~\eqref{eq:WC2d} obtained by turning the parameter $z_e$ into a dynamic variable. In these equations, the new parameter
$k$ controls the position of the equilibria of \eqref{WC3d}. One of the most common routes
to MMOs is the passage through a so-called Folded Saddle Node of type II (FSN II, see~\cite{mk_mw_10}). This transition is similar to canard explosion in the sense that it corresponds to a passage from a regime of a stable excitable equilibrium to MMOs or relaxation oscillators. We now investigate the presence of FSN II in our system~\eqref{WC3d} by considering the reduced system, obtained by setting $\eps=0$:
\begin{figure}[!t]
	\centering
		\includegraphics[width=\textwidth]{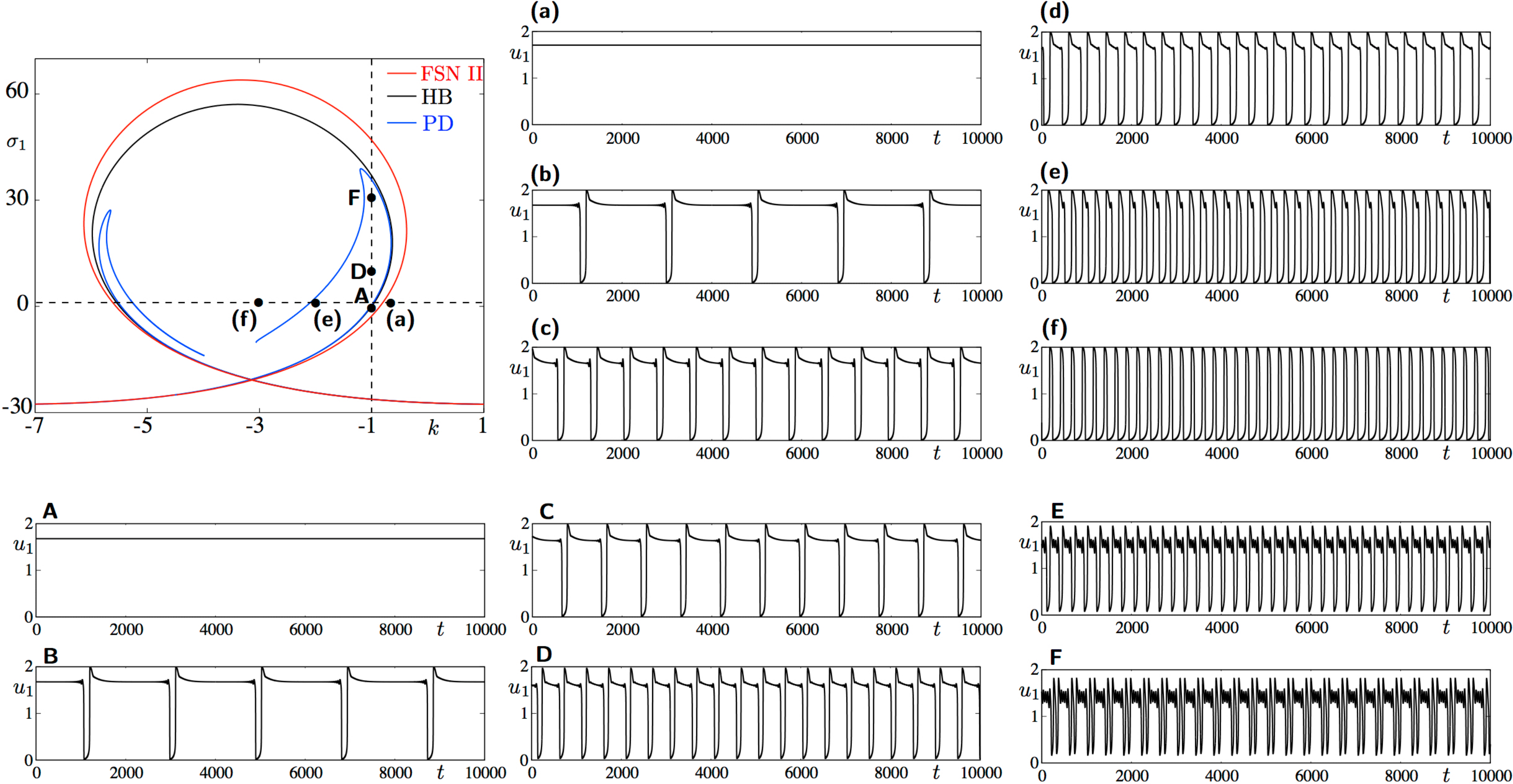}
	\caption{The top left panel shows a two-parameter bifurcation diagram of the 3D Wilson Cowan system~\eqref{WC3d} in the $(k,{\sigma}_1)$-plane, with Hopf (black) and period-doubling (blue) bifurcation branches.The red curve is the locus of FSN II bifurcation. All other panels show the long-term behavior of the system for particular values of $k$ with fixed ${\sigma}_1=0$ (panels (a) to (f)) and for particular values of ${\sigma}_1$ with fixed $k=-1$ (panels A to F), respectively.}
	\label{fig:drs}
\end{figure}
\begin{eqnarray}
0 &=& -u_1 + G_1(J_{11}u_1+J_{12}u_2+z_e)\label{rWC3d-u1}\\
   ~\dot{u_2} &=& -tu_2 + G_2(J_{21}u_1+J_{22}u_2)\label{rWC3d-u2}\\
   ~\dot{z_e} &=& k +\gamma z_e - u_1 - u_2\label{rWC3d-ze}
\end{eqnarray}
The constraint \eqref{rWC3d-u1} is eliminated by solving this algebraic equation in $u_2$. This provides an expression of $u_2$ as a function of $(u_1, z_e)$:
\[
u_2 = \frac{1}{J_{12}} \left(G_1^{-1}(u_1) - J_{11}u_1 - z_e\right) =: F(u_1,z_e,\sigma_1)
\]
and $G_1^{-1}$ is expressed in terms of the inverse of the $\erf$ function and of $\sigma_1$, dependence that is made explicit in $F$.  
We can now rewrite \eqref{rWC3d-u1}-\eqref{rWC3d-ze} in the form:
\begin{eqnarray}
		\nonumber\frac{\partial F}{\partial u_1}\dot{u_1} &=& -F(u_1, z_e, \sigma_1) + G_2(J_{21}u_1+J_{22}F(u_1, z_e, \sigma_1))\\
		 &&\qquad -\frac{\partial F}{\partial z_e}(k +\gamma z_e - u_1 -F(u_1, z_e, \sigma_1)) \label{reWC3d-u1}\\
		   ~\dot{z_e} &=& k +\gamma z_e - u_1 - F(u_1, z_e, \sigma_1)\label{reWC3d-ze}
   \end{eqnarray}
\begin{figure}[!t]
	\centering
		\includegraphics[width=.8\textwidth]{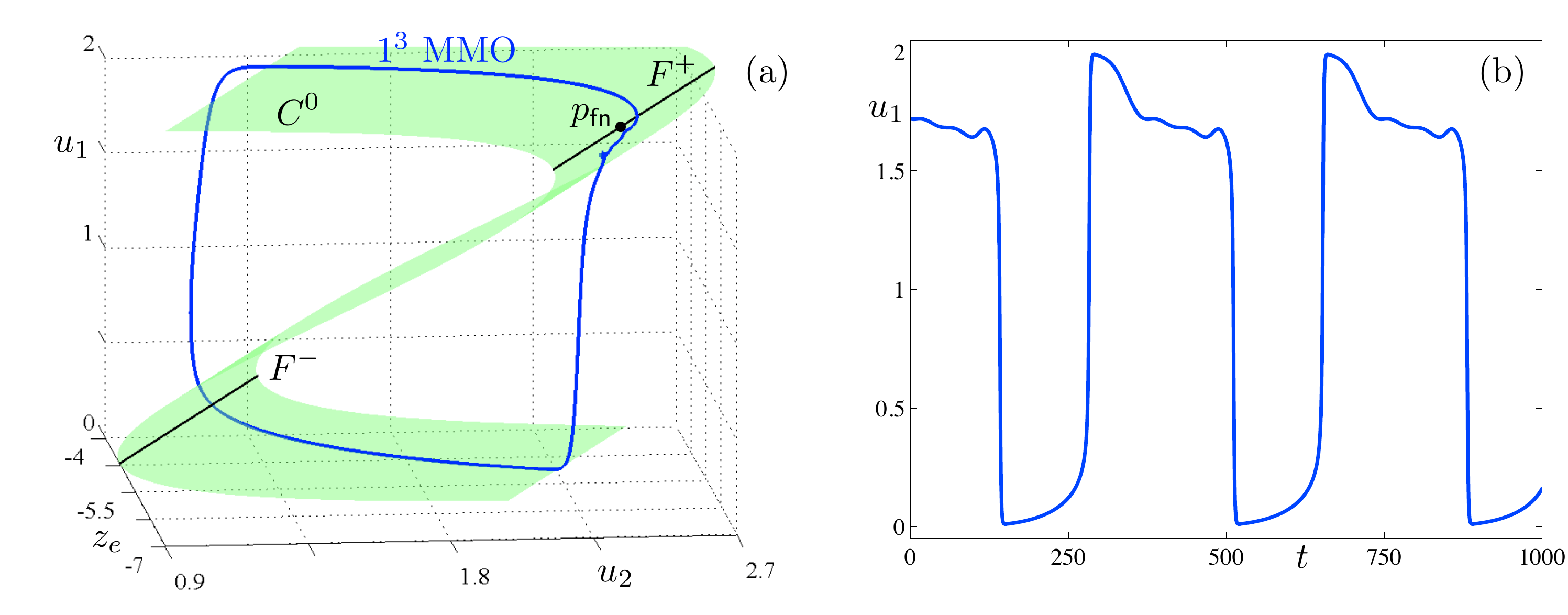}
	\caption{{Panel (a)} shows the critical manifold $C^0$ of system~\eqref{WC3d} (green surface) viewed in the three-dimensional phase space $(u_1,y,z_e)$. Also shown are both the fold curves $F^{\pm}$ (black lines) and an MMO of type $1^3$, solution to the problem, whose existence can be explained by the presence of a folded node $p_{\mathrm{fn}}$ (black dot) in the system. {Panel (b) shows the time profile of $u_1$ for this MMO}.}
	\label{fig:critmanmmo}
\end{figure}
Now the fold line for system \eqref{reWC3d-u1}-\eqref{reWC3d-ze} corresponds to the equation $\frac{\partial F}{\partial u_1}=0$
and folded singularities correspond to points on the fold line for which the RHS of \eqref{reWC3d-u1} is $0$.
An FSN II point is a point on the fold line for which the RHS of both \eqref{reWC3d-u1} and \eqref{reWC3d-ze} are zero.
Such points are of codimension 1 in the parameter space, and, under some non-degeneracy conditions,
correspond to a passage of a true equilibrium of \eqref{reWC3d-u1}-\eqref{reWC3d-ze} through the fold line~\cite{mk_mw_10}.
As a parameter goes through an FSN II point the system passes from a regime of excitability (stable equilibrium near the fold)
to a regime of MMOs or relaxation\footnote{In fact other types of dynamics
could occur, but a passage to MMOs or relaxation is what is often observed.}, depending on the global part of the flow \cite{mb-mk-mw_06}. On the level of folded singularities, FSN II corresponds
to a passage from folded saddle to folded node, and folded nodes are associated to the occurrence of MMOs \cite{mw_05, mb-mk-mw_06, md-jg-bk-ck-ho-mw_12}. Note eventually, as demonstrated in~\cite{jg_08}, the presence of a curve of Hopf bifurcations in the vicinity of an FSN II point. 

This is precisely the scenario of transition occurring in out system. Indeed, in Figure \ref{fig:drs}, we display in the upper left corner the FSN II curve (shown in red), computed using MATLAB based on the above formulae. Nearby (in black) is the curve of Hopf bifurcations (HB) for $\eps=0.005$, which, as expected, is a close approximation of the FSN II curve. 

Fine analysis of the passage from excitability to MMOs shows a typical rich bifurcation structure \cite{md-jg-bk-ck-ho-mw_12,jg_08}. This structure, occurring
as a result of changing the parameter $k$, is highlighted in figure~\ref{fig:mmo}. In particular, we find a cascade of period-doubling bifurcations (PD) (best seen on panel (b), which is an enlargement of panel (a)) and a family of complicated closed branches or \textit{isolas} (we computed four of them). This structure is compatible with the analysis done in~\cite{jg_08} on a minimal system of a singular Hopf bifurcation. 

\begin{figure}[!t]
	\centering
		\includegraphics[width=\textwidth]{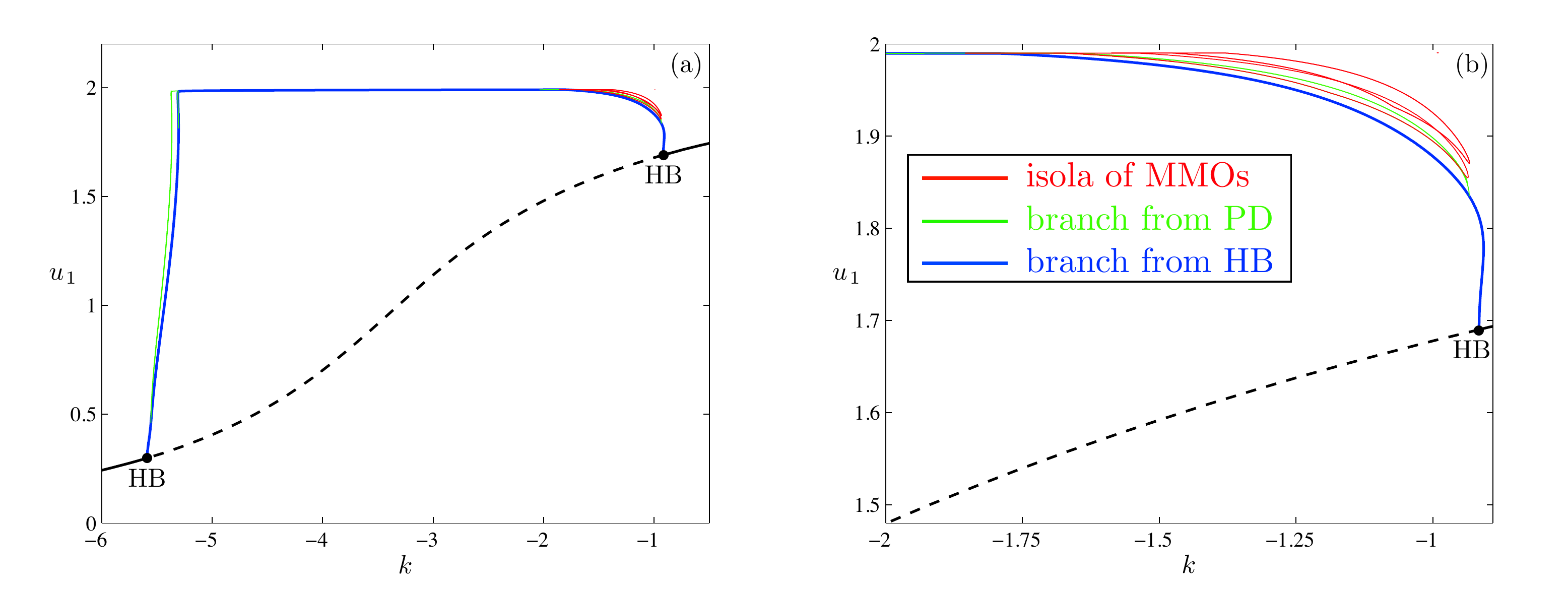}
	\caption{Bifurcation diagram of the original system~\eqref{WC3d} as a function of parameter $k$. Panel (a) shows a large view of the branch of equilibria (black) and the branches of periodic solutions (blue) born at two Hopf bifurcations (black dot) that are symmetric with respect to $k=0$. Panel (b) is an enlarged view in the parameter region where additional branches of periodic solutions exist, both branches born at multiple period-doubling bifurcations ($P\!D$) as well as closed branches (isolas). The vertical axis correspond to the computed $L_2$-norm of the solutions.}
	\label{fig:mmo}
\end{figure}

Changes in the parameters $(k,\sigma_1)$ correspond to paths in this diagrams governing the bifurcation scenarios. Two particular paths are displayed in Figure \ref{fig:drs}: (i) varying $k$ while keeping $\sigma_1=2$ fixed (small letters) and (ii) fixed $k=-1$ and $\sigma_1$ varying (capital letters), and a phase space view of an MMO trajectory is given in Figure \ref{fig:critmanmmo}. Multiple MMO traces corresponding to these two variations of parameters highly the typical traces found inside the closed curve corresponding to FSN II. The scenario (i) will be explored numerically in section~\ref{sec:Network} in the network equations, and scenario (ii) highlights the fact that for $k$ fixed and $\sigma$ increasing from $0$ to a positive value, we observe a transition from excitability to MMOs as a noise induced bifurcation. 

Therefore, the analysis of the system of ODEs arising from taking the average of the solution to the mean-field equations in Wilson-Cowan systems with slow adaptation currents yields to the conclusion that (i) qualitative behaviors sensitively depend upon noise levels, as we already learnt from the analysis of section~\ref{wc2d} and (ii) that very delicate phenomena such as canard-induced MMOs persist in the presence of possibly large noise. This phenomenon is, again, very surprising, and may only be understood in the context of collective effects in very large networks. Indeed, large noise of course will destroy the fine structure of small oscillations in regimes where the noiseless system presents mixed-mode oscillations, but we observe here that these persist for large noise as a collective effect. Even more interesting is that noise actually induces the apparition of such solutions, and that the shape of the solution sensitively depends on the precise noise levels neurons are submitted to. 

\section{Back to the network dynamics}\label{sec:Network}
Thus far, we showed that in the limit $N\to\infty$, the network equations~\eqref{eq:WCNet} can be reduced to mean-field equations whose mean satisfy an ordinary differential equation where the noise parameter, appearing as a parameter of the limit equation, is related to the presence of canards and MMOs. We now return to the original network equations~\eqref{eq:WCNet} with finite $N$. We show that the predicted presence of complex oscillations as well as canard explosion are generally recovered in the simulations of large but finite networks. In certain parameter regions, systematic errors related to the finiteness of the number of neurons considered however appear. These finite-size effects, corresponding to the presence of random MMOs in two-populations networks (in which cases no MMO is present in the mean-field limit) are analyzed in view of the competition between multiple attractors of the mean-field limit and finite-size fluctuations around that limit. In three populations networks, the MMO solutions are generally very fragile close from the transitions, and the finiteness of the number of neurons generally induces early jumps shortening the duration of the small oscillations phase. 

The simulation of large networks presents a number of theoretical and technical obstructions. Indeed, at least three difficulties arise: the stiffness of the solutions and the large periods of relaxation oscillations related to the slow-fast nature of the equations, the dimensionality issue related to the number of neurons considered ($N$), and the stochastic nature of the trajectories. The stiffness issue is generally taken care by using refined simulation algorithms, that generally do not have a stochastic counterpart. Here, accurate simulations will necessitate to use small timesteps $dt$ and large simulation times $T$ (due to the typical periods of relaxation oscillations), which will raise two difficulties: the computation time will increase proportionally to $T/dt$, and the number of independent random variables we need to draw during the simulation will increase proportionally to $NT/dt$, which can raise concerns about the fact that the correlations of the pseudo-random sequence of random numbers may surface. Moreover, with large simulation times, error accumulation will arise, proportionally to the computation time. With these limitations in mind, we developed a vectorized simulation algorithm based on Euler-Maruyama scheme implemented in Matlab\textregistered, which remains relatively efficient (computation times are almost independent of the network size) at the levels of precision needed. Typically, the simulation of a $N=2\,000$ neurons network over a time interval $T=1\,500$ with timestep $dt=0.01$ will take $66s$ on a Macbook Pro with processor 2.3 GHz intel core i7 with 16GB 1600 MHz DDR3. The specific development of computational algorithm, precise analysis of the error ad limitations raised by the number of random numbers drawn are not in the scope of the present paper, and we are confident that our implementation, with proper choices of timestep, total times and number of neurons, provide accurate descriptions of the network behavior in the examples provided below. 

\subsection{Canard Explosion in the 2 populations Wilson-Cowan system}
We start by analyzing the finite-size effects in our two population network. For our choice of parameters, we have seen that a subcritical Hopf bifurcation and canard explosion occurs when varying the intensity of the noise (see Fig.~\ref{fig:BifDiagColors2d}). The full bifurcation diagram, plotted in Fig.~\ref{fig:wc2d} displayed a canard explosion with a branch of limit cycles persisting until very large values of the noise standard deviations (above 45). Such standard deviations correspond to extremely large noise that are generally very hard to simulate accurately and seem unreasonably large in the Wilson-Cowan model. A restricted bifurcation diagram is displayed in Fig.~\ref{fig:BifDiagColors2d} and shows that, as a function of noise levels, the system either presents a unique stable stationary distribution (blue region), an unstable fixed point and a stable periodic orbit corresponding to relaxation oscillations for large noise (pink region), and an intermediate noise region (yellow) where we have co-existence of a stable stationary distribution and a stable periodic orbit, separated by an unstable periodic orbit. This parameter region is referred to as the bistability regime. 

In the case where a unique stable solutions exists for the limit averaged equation (pink or blue regions), network trajectories will display random fluctuations well approximating the attractor. This is what is plotted in Fig.~\ref{Fig:UniqueAttractor2d}, both in the case of stationary and periodic solutions. The canard explosion, displayed in section~\ref{sec:Bifurcations}, corresponds to the sharp transition from fixed point to relaxation oscillations, and due to the fact that the canard explosion occurs through subcritical Hopf bifurcation prevents from observing the full canard explosion and canard cycles, that correspond to unstable trajectories of the limit ODE and of the network equations.

\begin{figure}[htbp]
	\centering
		\subfigure[Bifurcation Diagram]{\includegraphics[width=.3\textwidth]{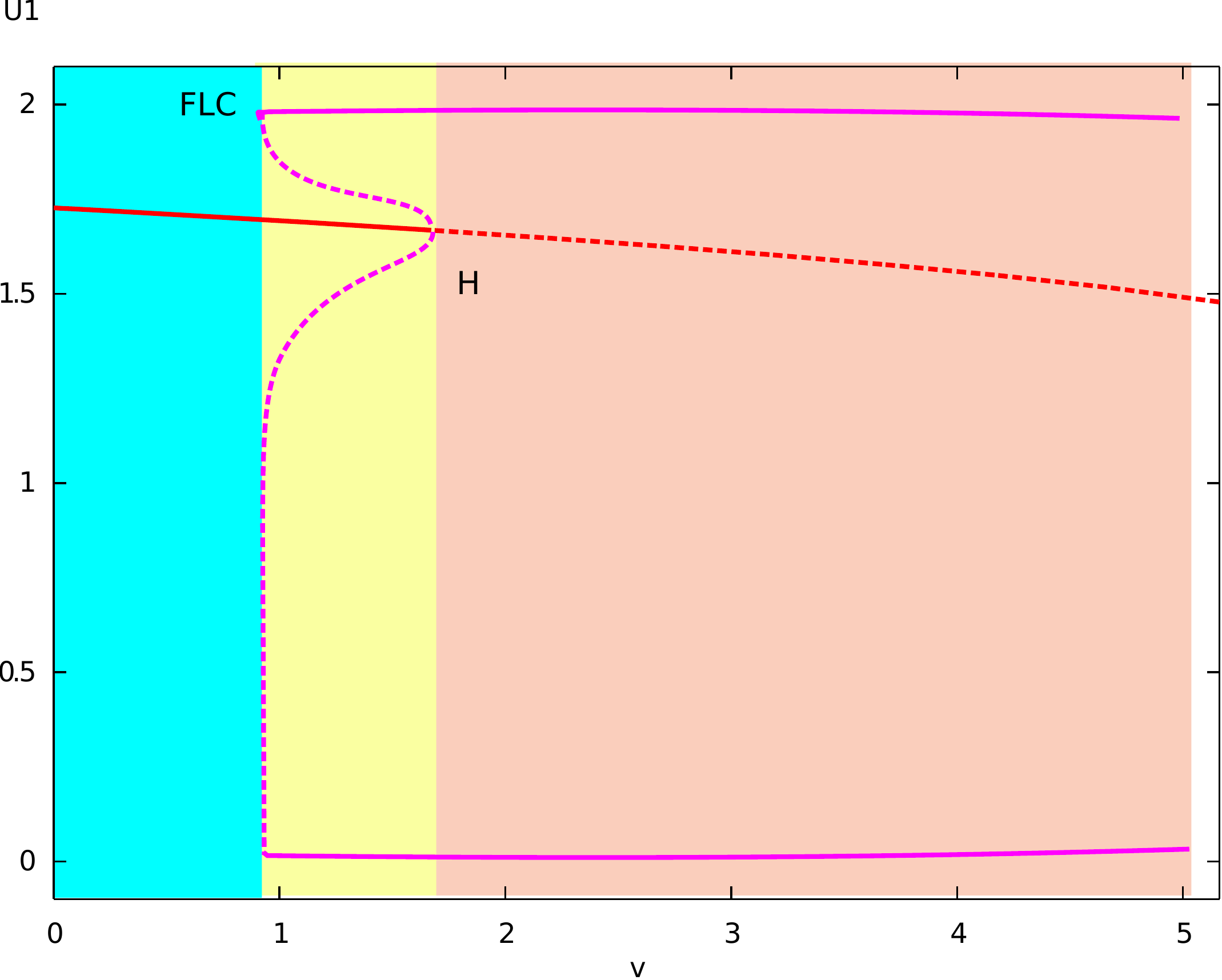}\label{fig:BifDiagColors2d}}
		\subfigure[Stationary regime, $v=0.5$]{\includegraphics[width=.65\textwidth]{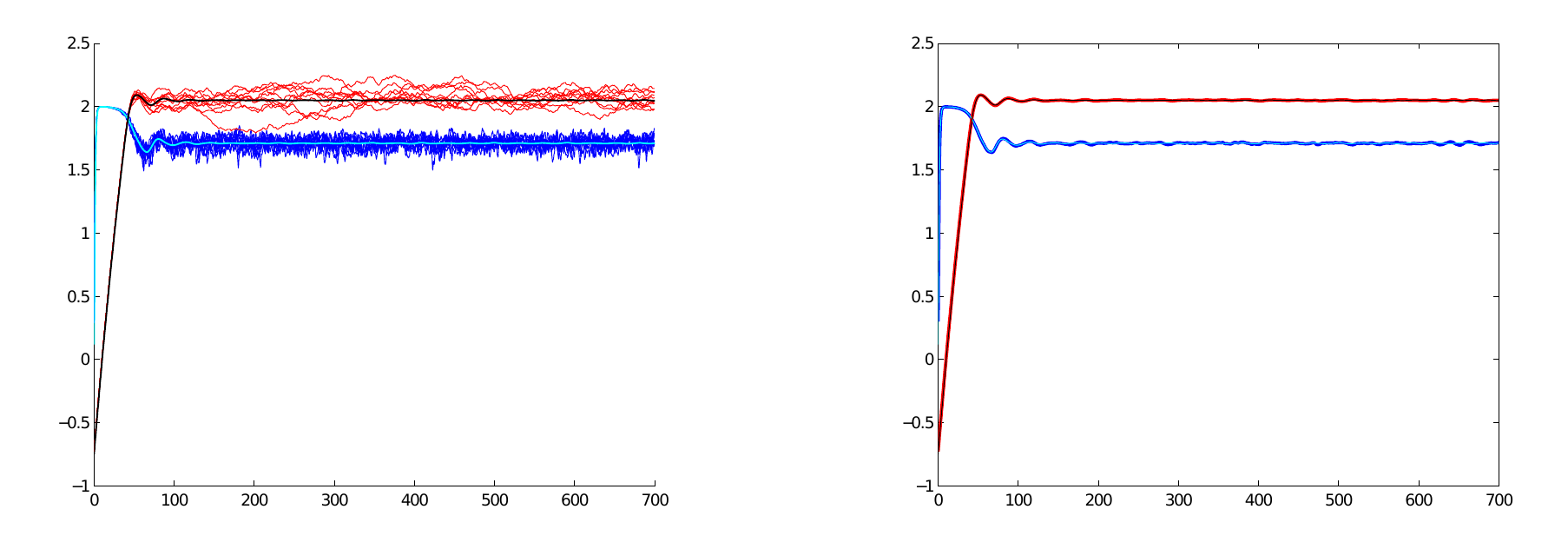}\label{fig:StationRegime2d}}
		\subfigure[Periodic Regime, $v=2$]{\includegraphics[width=.65\textwidth]{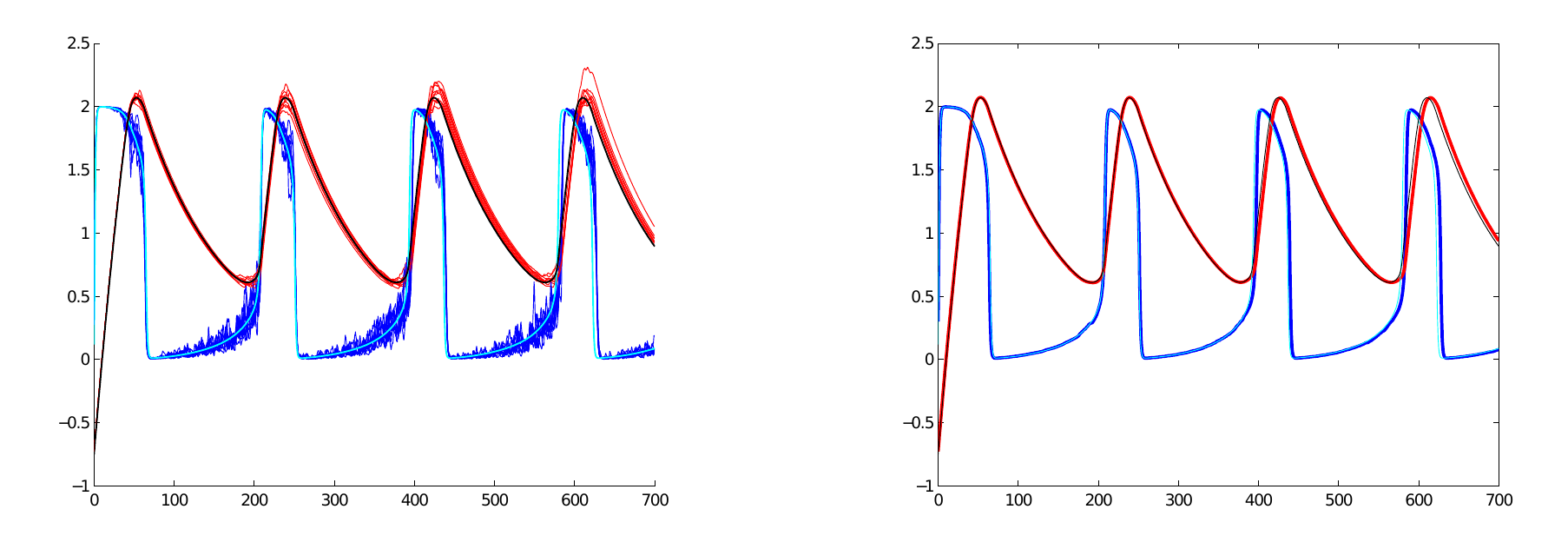}\label{fig:PeriodicRegime2d}}
	\caption{Network Simulations for the Wilson-Cowan two populations system, whose asymptotic dynamics was analyzed in~\ref{fig:wc2d}. (a) zoomed bifurcation diagram for relevant noise intensities, (b) and (c) are examples of trajectories of the network equations in the cyan and pink regions respectively. (b): stationary solutions, $v=0.5$: (left) 10 randomly chosen trajectories of the network (blue: population 1, red: population 2) display random fluctuations of the around the predicted stationary average (cyan: population 1 and black: population 2) computed through the limit ODE. (right) empirical average vs theoretical limit average, show a very precise fit. (c) Same simulation and color code as in (b) but for $v=2$ corresponding to the oscillatory regime (pink region). Simulation parameters: $2\,000$ neurons per population, Euler-Maruyama scheme with $dt=0.01$ and randomly chosen initial conditions. }
	\label{Fig:UniqueAttractor2d}
\end{figure} 

In the bistability regime, trajectories are more complex and finite-size effects will result in random switches between the two attractors. Typical shape of the phase plane, displayed in Fig.~\ref{fig:PhasePlane2d}, show the coexistence of relaxation oscillations and stable fixed points. The separatrix between the two attractors is the unstable limit cycle arising from the Hopf bifurcation, winding around the fixed point. Incidentally, we observe that the unstable periodic orbit and the relaxation oscillations cycles are extremely close (they are arbitrarily close when the intensity of the noise approaches the value corresponding to the fold of limit cycles). Therefore, in this parameter region, random switches from relaxation oscillations to random fluctuations around the stable fixed point will occur, phenomenon that we term random attractor switching. 
\begin{figure}[htbp]
	\centering
		\subfigure[Phase plane, $v=1.1$]{\includegraphics[width=.4\textwidth]{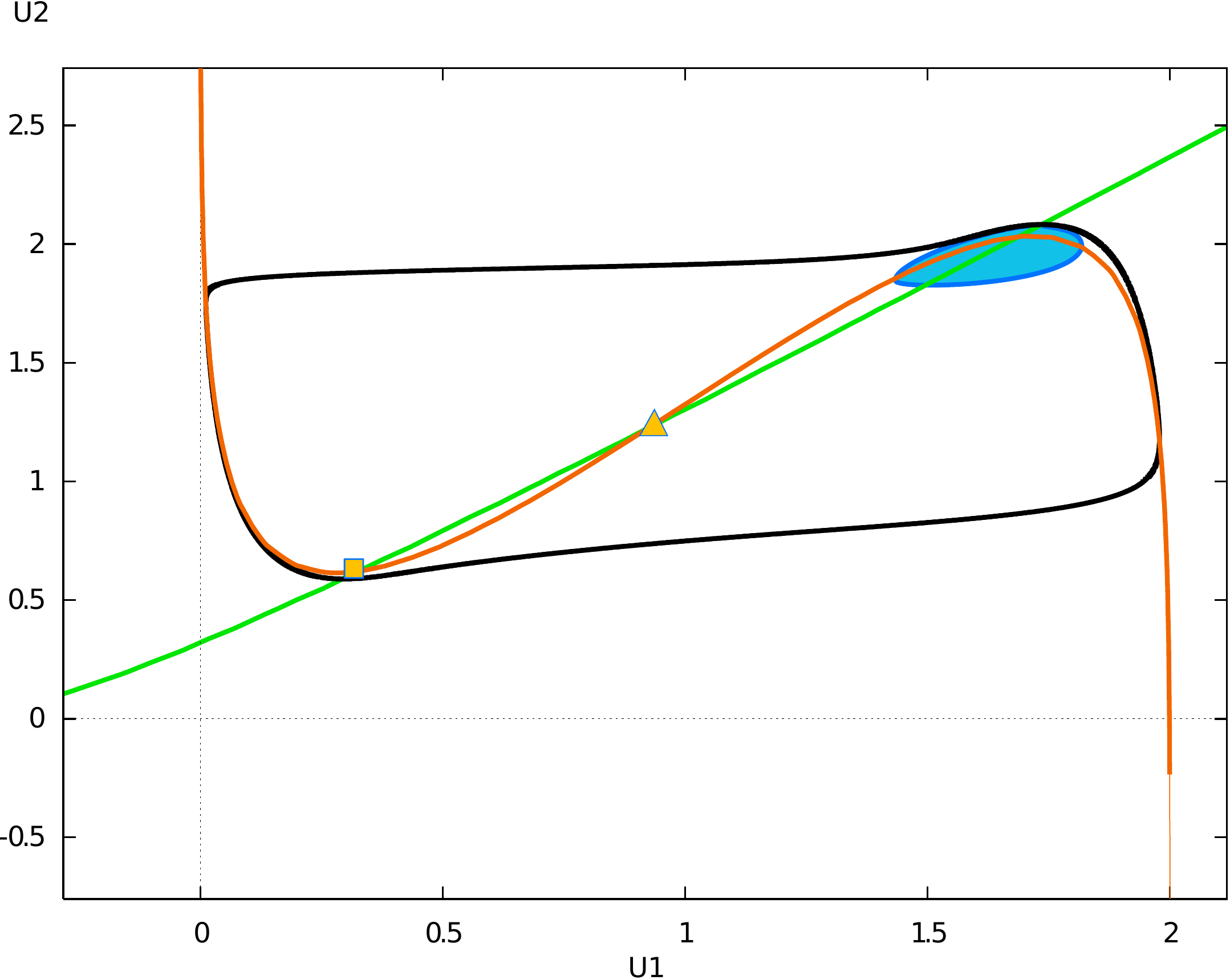}\label{fig:PhasePlane2d}}
		\subfigure[Trajectories, $v=1.1$]{\includegraphics[width=.7\textwidth]{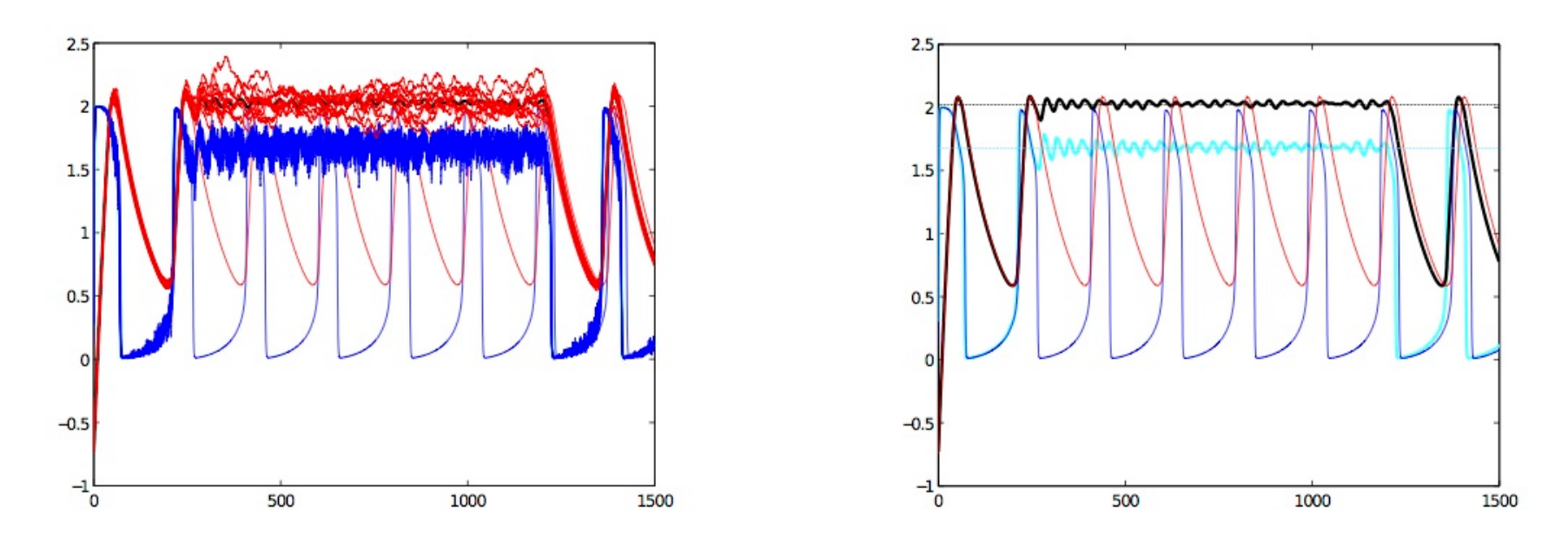}\label{fig:TrajSwitches}}
	\caption{Bistable regime in the two-populations Wilson-Cowan system. (a) Geometry of the phase plane of the limit ODE, shows the presence of a relaxation oscillations cycle (black) and an unstable cycle (blue) delineating the attraction basin of the stable fixed point. Red and Green curves are the nullclines of the system, the triangle (square) depicts a saddle (unstable focus) fixed point. (b) depicts a typical trajectory of random attractor switching, with $1\,000$ neurons per population and $v=1.1$, same color code as in Fig.~\ref{Fig:UniqueAttractor2d}. Dotted lines depict the location of the stationary mean for both populations.}
	\label{fig:Bistability}
\end{figure}

Similarly to what was shown in~\cite{berglund-landon:12}, we may interpret the switching trajectories as random MMOs. However, in contrast to their setting, the noise perturbation involved in our case are not driven by white noise with constant standard deviation. As we analyze in a forthcoming study, the finite-size perturbations correspond to Gaussian perturbations with standard deviation function of the solution of the mean-field equations. The nature of this process implies that the random switches do not occur completely randomly, but depend on the size of the fluctuations, which is governed by a function of the mean-field equation. A consequence will be, as presented in the example plotted in Fig.~\ref{fig:TrajSwitches}, that the system has a bias to switch in phase with the virtual periodic attractor. This surprising phenomenon will be analyzed in greater detail in a forthcoming study. 

Let us eventually discuss the occurrence of switches as a function of the size of the network and of the parameters. We observe that the attraction bassin of the stationary solution becomes increasingly large as noise approaches the value corresponding to the fold of limit cycles, and increasingly small when approaching the Hopf bifurcation. Close from the value of the fold of limit cycles, random switches from the periodic to the stationary solutions will become very probable and switches from the stationary to the periodic solutions are less probable, implying that the averaged time spent in the neighborhood of the stationary solution will become larger as we approach the fold of limit cycle. A symmetrical phenomenon will arise close from the Hopf bifurcation. In order to illustrate this phenomenon, we display in Fig.~\ref{fig:AverageTime} the average time spent in a neighborhood of the stationary solution (actually inside the deterministic attraction bassin of the stable fixed point depicted in blue in Fig.~\ref{fig:PhasePlane2d}) as a function of the noise level, averaged across 20 trajectories. These switching probabilities also depend on the size of the network. Indeed, as the number of neurons tends to infinity, that probability tends to zero, and the trajectory will be deterministically locked to that predicted by the mean-field theory and the initial condition, and the smaller the network size, the more likely the switches will occur. 

\begin{figure}[htbp]
	\centering
		\includegraphics[width=.3\textwidth]{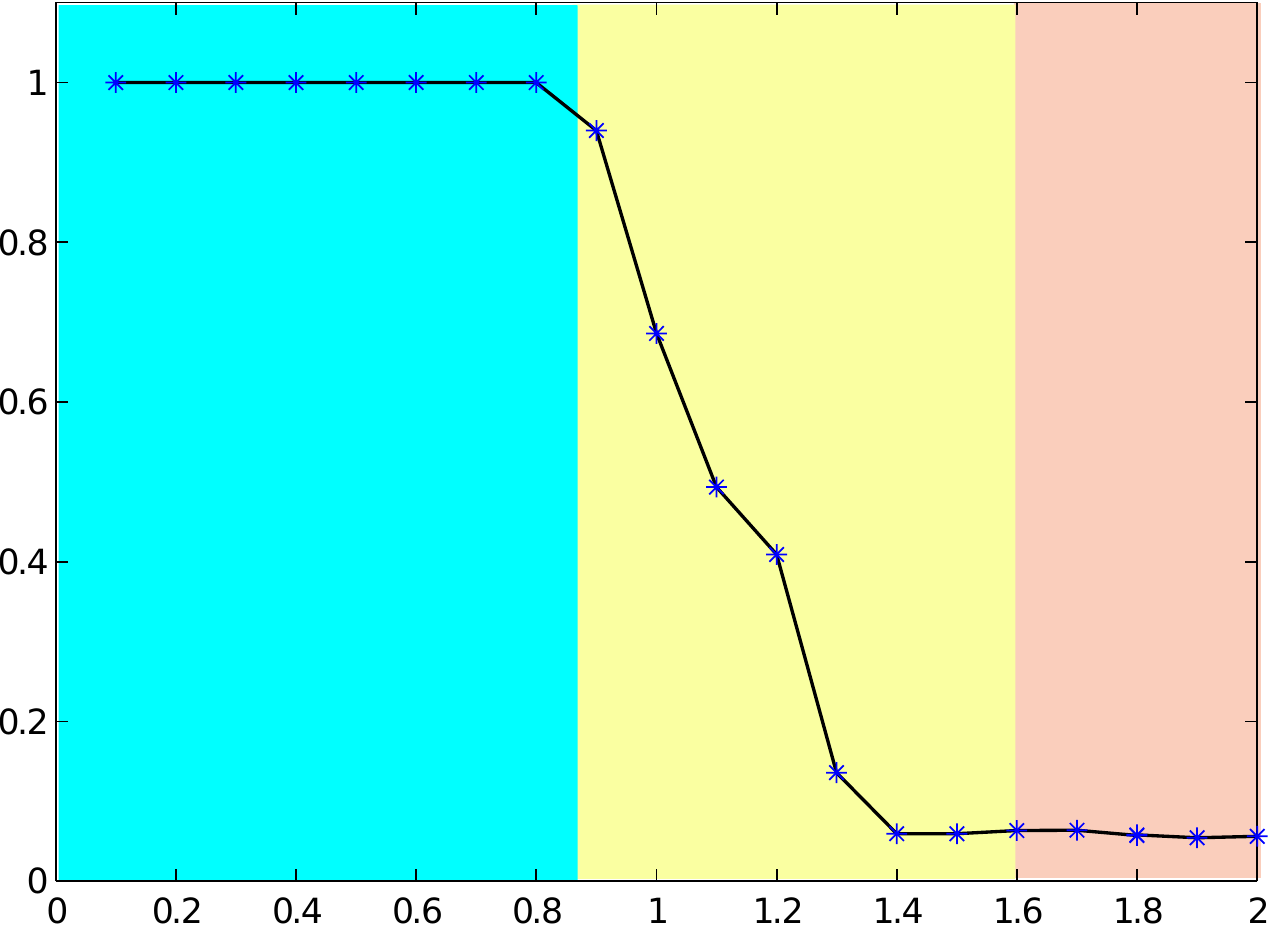}
	\caption{Proportion of time along trajectories spent in the vicinity of the stationary solution as a function of noise levels. Colors correspond to the regions in the bifurcation diagram~\ref{fig:BifDiagColors2d}. }
	\label{fig:AverageTime}
\end{figure}

We therefore conclude from the present analysis that finite-size effects barely affect the solution, and that the mean-field model accurately describes the behavior of the network in most of the parameter space, except during the canard explosion, in an exponentially small region of parameters, and in that region, random switches between multiple attractors may occur, as a function of the network size and the geometry of the limit dynamical system. 

\subsection{Mixed-Mode Oscillations: early jumps and hidden small oscillations}
We now investigate the behavior of finite-sized networks in the 2 populations Wilson-Cowan system with slow adaptation currents analyzed in the mean-field limit in section~\ref{wc3d}. The analysis the resulting mean-field system showed that the system presented a FSN II bifurcation related to the presence of mixed-mode oscillations, the shape of which sensitively depended of the noise levels $\sigma_1$ and $k$, as depicted in Fig.~\ref{fig:drs}. The reduction obtained of these equations allowed analyze the behavior for extremely high values of the noise which are not accessible through numerical simulations of the noisy network, chiefly due to precision constraints as we mentioned.

We now confront the different behaviors observed in section~\ref{wc3d} to simulations of the noisy network equations with the same parameters. We expect in this case that the trajectories obtained in the MMO regime to be less resilient to perturbations than the stationary and periodic orbits analyzed in the two-populations network, due to the sensitivity of the attractor in particular in the small oscillations phase. In contrast with the analysis of the two-populations Wilson-Cowan model, we will therefore observe more drastic finite-size effects even in regimes where there exists a unique attractor, since it is the very structure of the attractor that presents sensitivity to perturbations. These will lead to escape the attractor during the small oscillations phase, or \emph{early jumps}. 

This is what we observe in Fig.~\ref{fig:MMOsNet} where we reproduce the bifurcation scenario displayed in Fig.~\ref{fig:drs} for fixed noise $\sigma_1=2$ and different values of the parameter $k$ (the only scenario we did not display was (B), which requires very simulations on very long time intervals due to the very large period of the mixed-mode oscillation). In all cases, we recover the qualitative nature of the solutions predicted in the infinite-size limit, and early jumps, corresponding to systematic random escape of the attractor, occur due to finite-size effects. Moreover, we observe that the small oscillations tend to be hidden by the noise related to the finite-size effects. 
\begin{figure}[htbp]
	\centering
		\subfigure[$k=-0.8$]{\includegraphics[width=.3\textwidth]{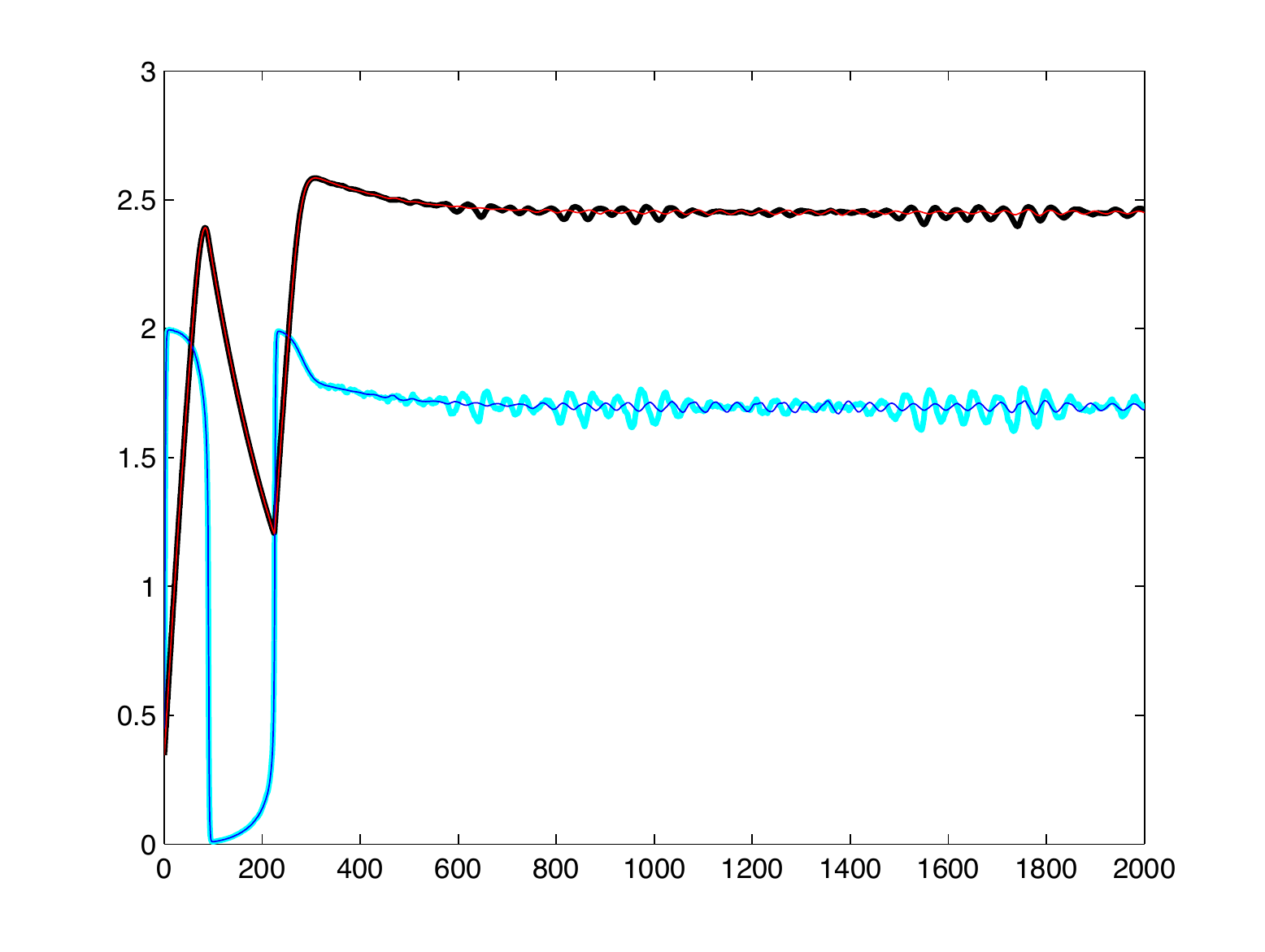}}
		\subfigure[$k=-1.2$]{\includegraphics[width=.3\textwidth]{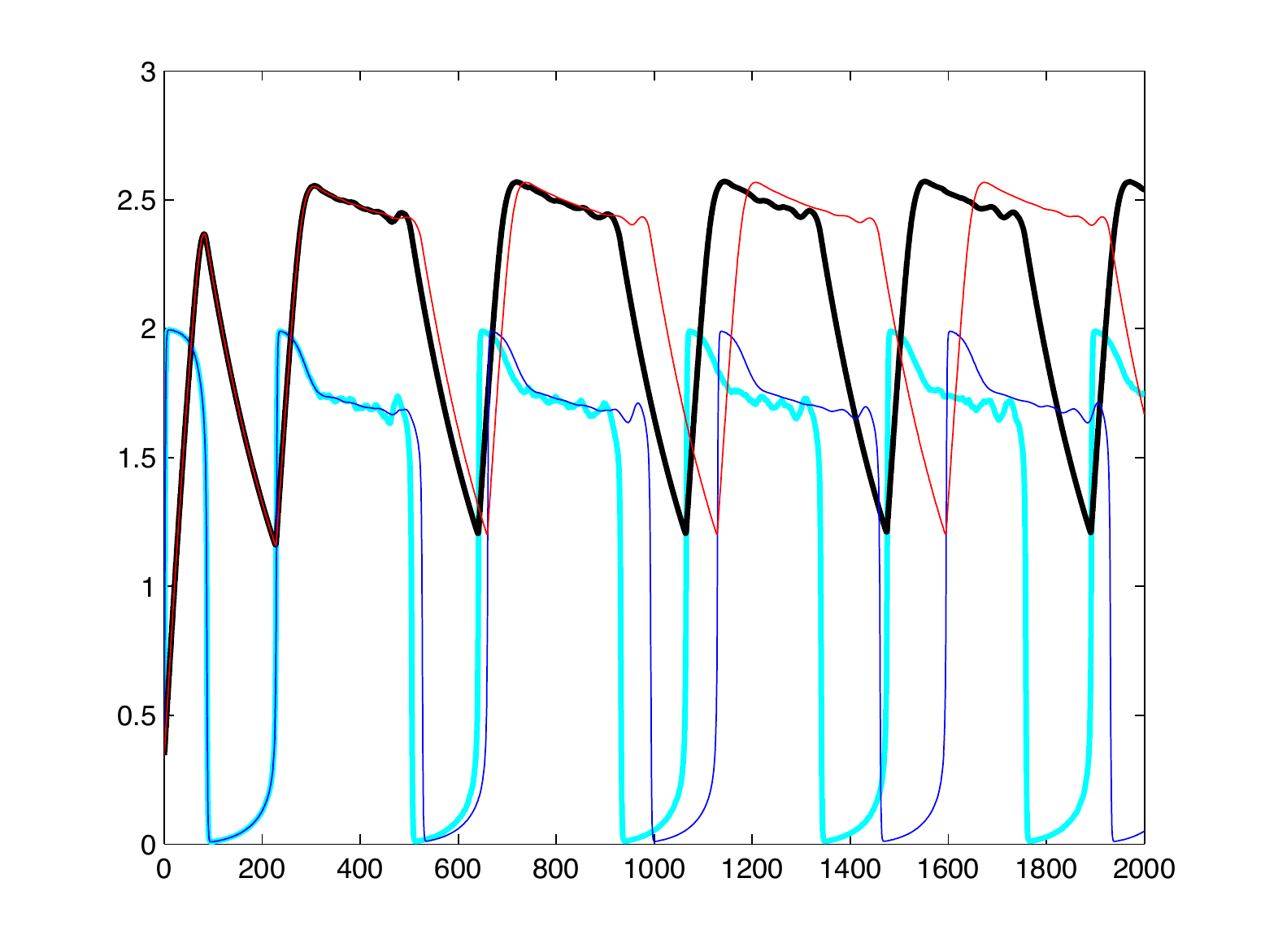}}
		\subfigure[$k=-1.5$]{\includegraphics[width=.3\textwidth]{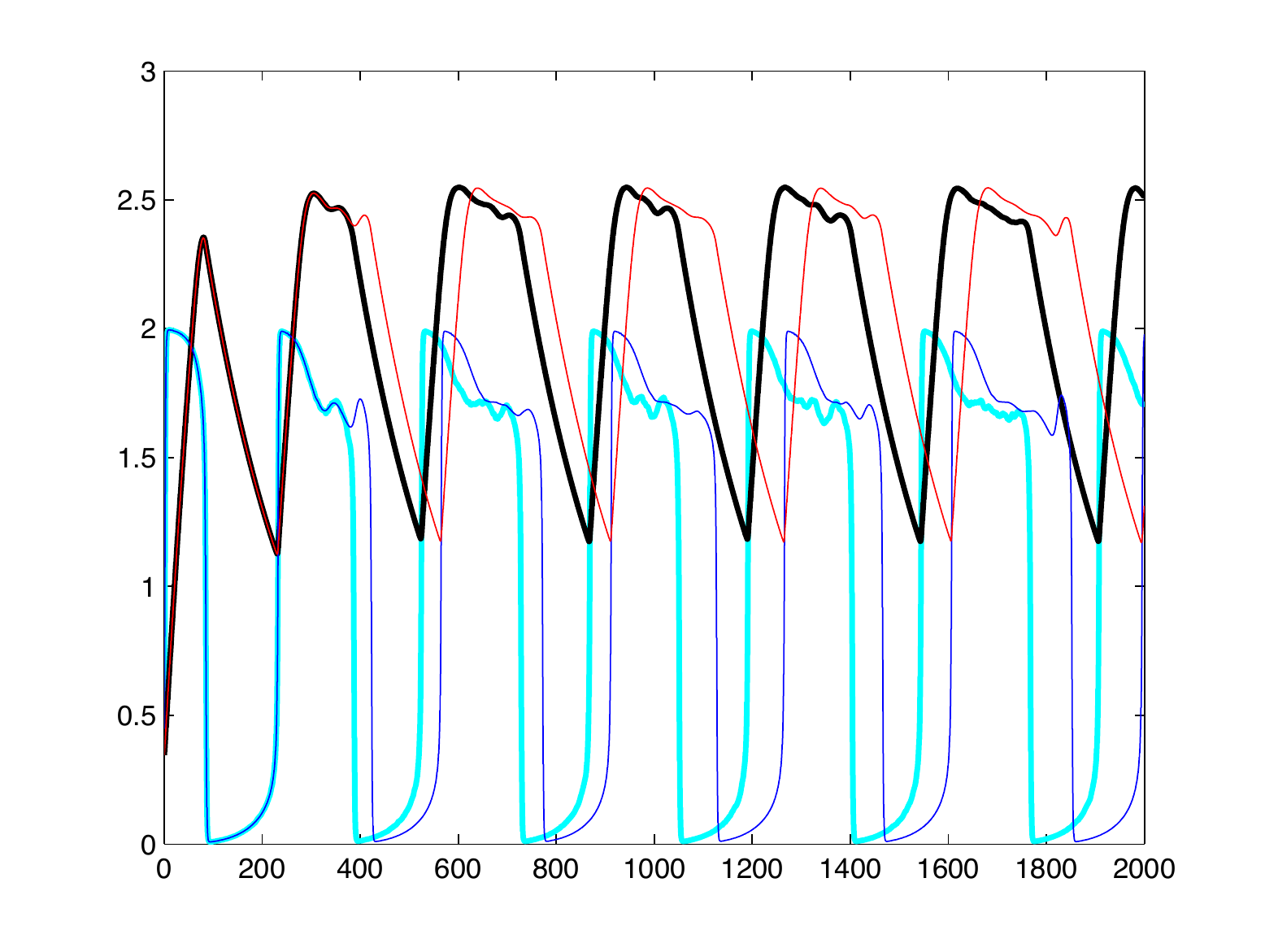}}
		\subfigure[$k=-2$]{\includegraphics[width=.3\textwidth]{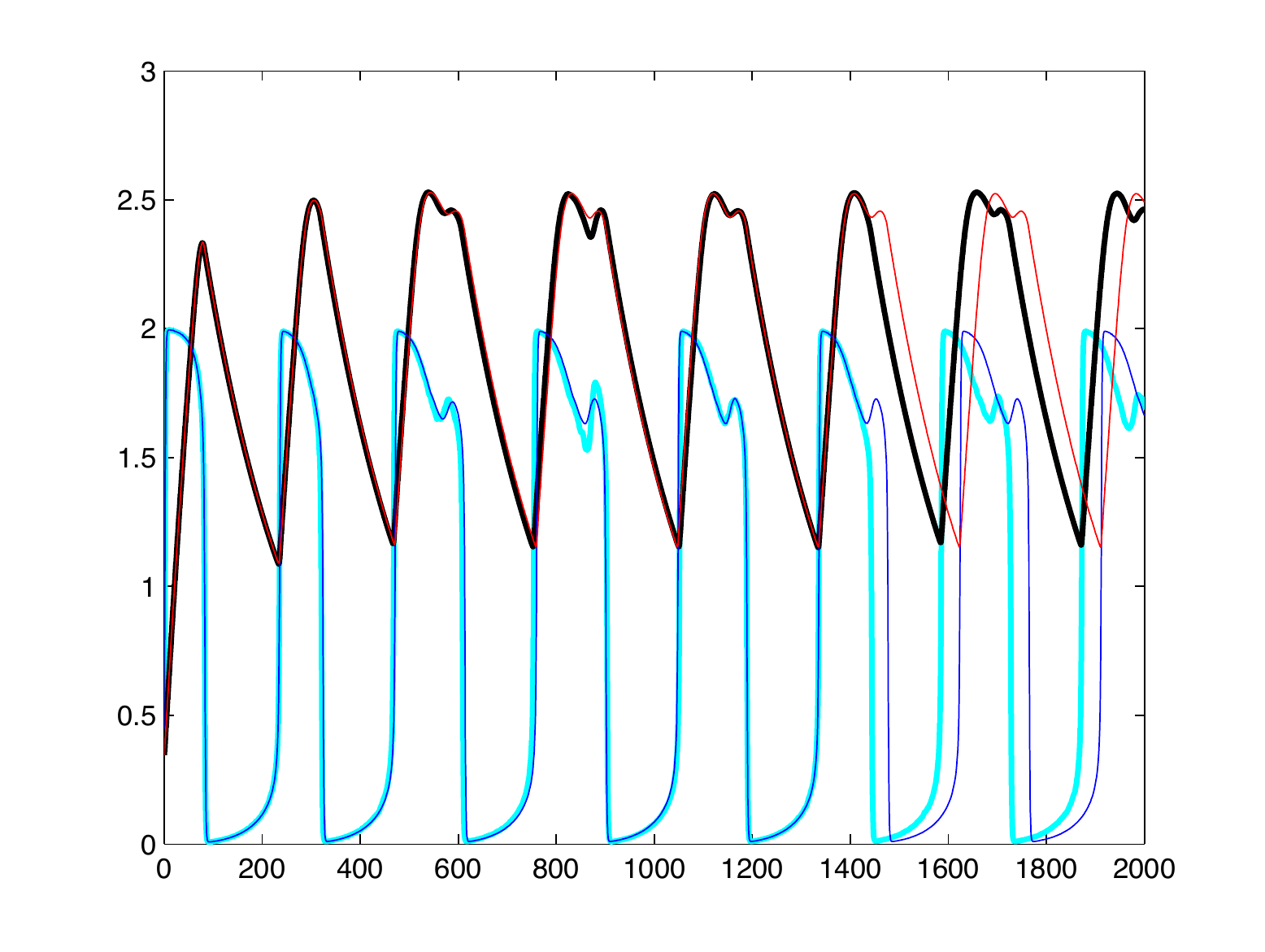}}
		\subfigure[$k=-3$]{\includegraphics[width=.3\textwidth]{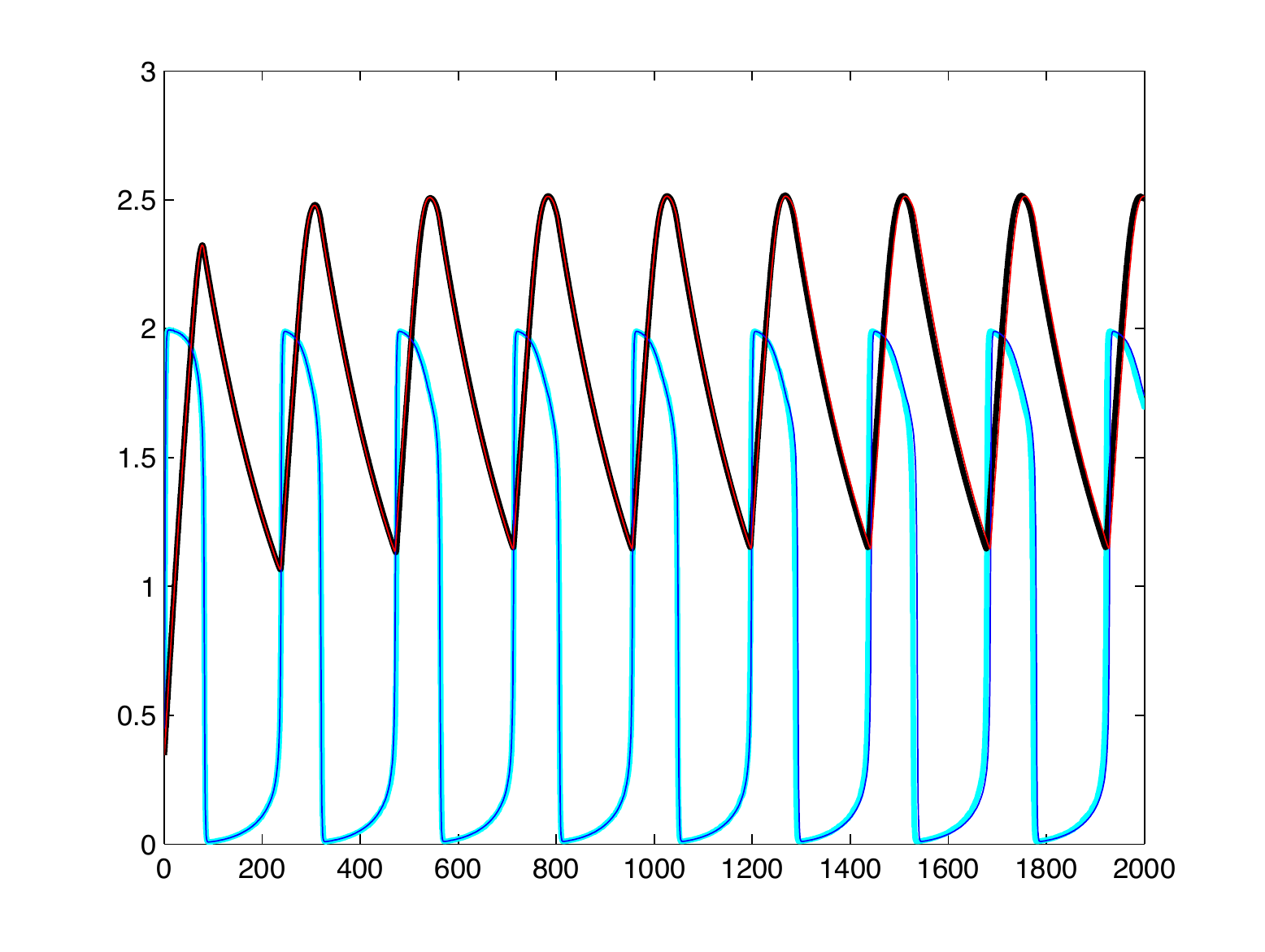}}
	\caption{Solution to the mean-field equations~\eqref{eq:WCMFE} (red and blue lines) and empirical averages of the network equation for $N=2\,000$ neurons in the cases (a), (c), (d), (e) and (f) of Fig.~\ref{fig:drs}. In each case, we observe that qualitative behaviors are recovered in the network equations, and that early jumps occur systematically.  }
	\label{fig:MMOsNet}
\end{figure}
These two finite-size effects of course disappear when considering larger networks and one recovers with precision the attractor predicted by the theory, as we show in Fig.~\ref{fig:MMOPrecise} for a network of $N=10\,000$ neurons. But as soon as the number of neurons is smaller, the network will systematically fail to follow precisely the predicted MMO and will generally escape the attractor prior than the infinite-sized system. 
\begin{figure}
	\centering
	\subfigure[$k=-1.5$, $v_1=2$]{\includegraphics[width=.45\textwidth]{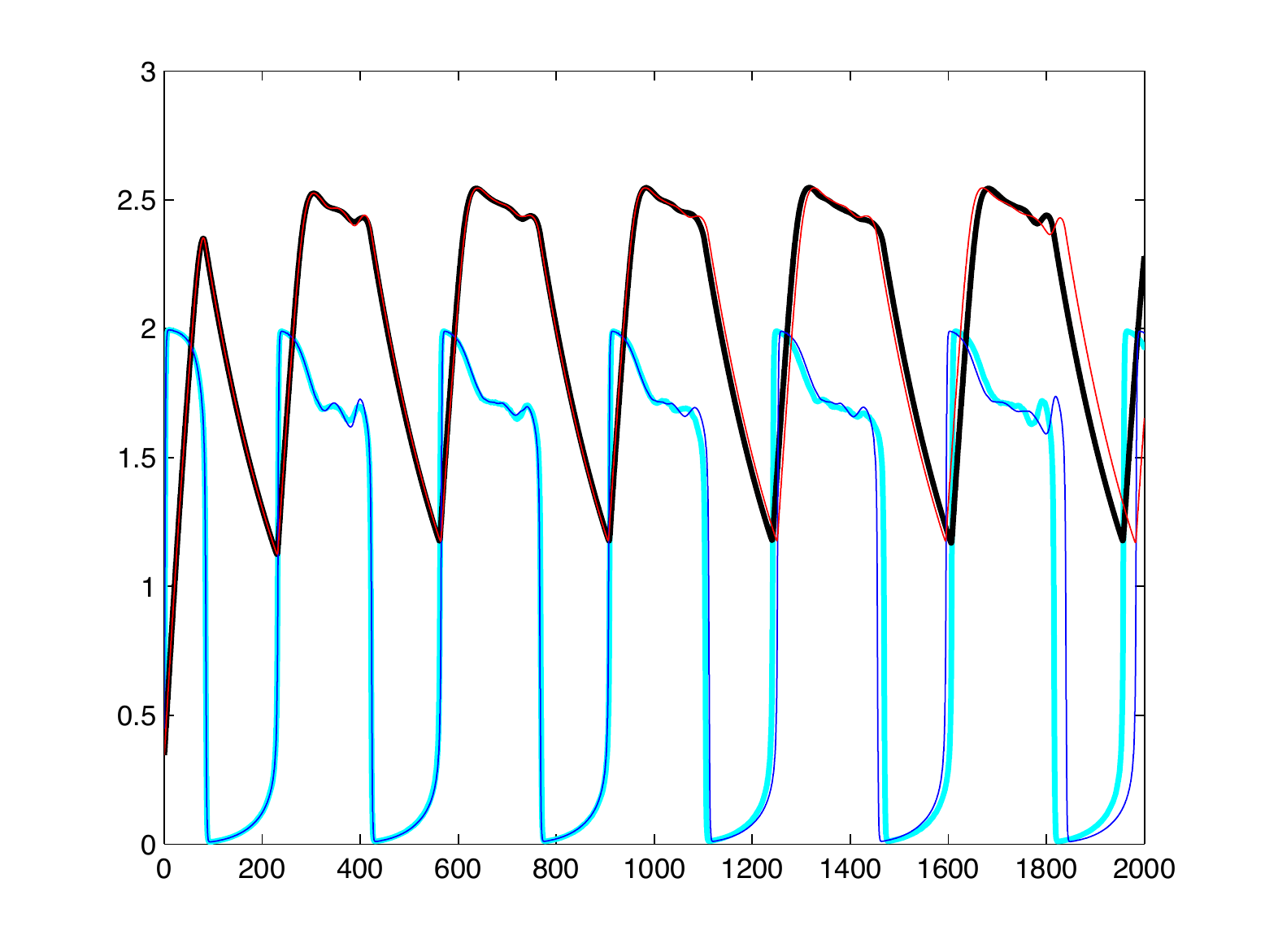}}
	\subfigure[$k=-1.1$]{\includegraphics[width=.45\textwidth]{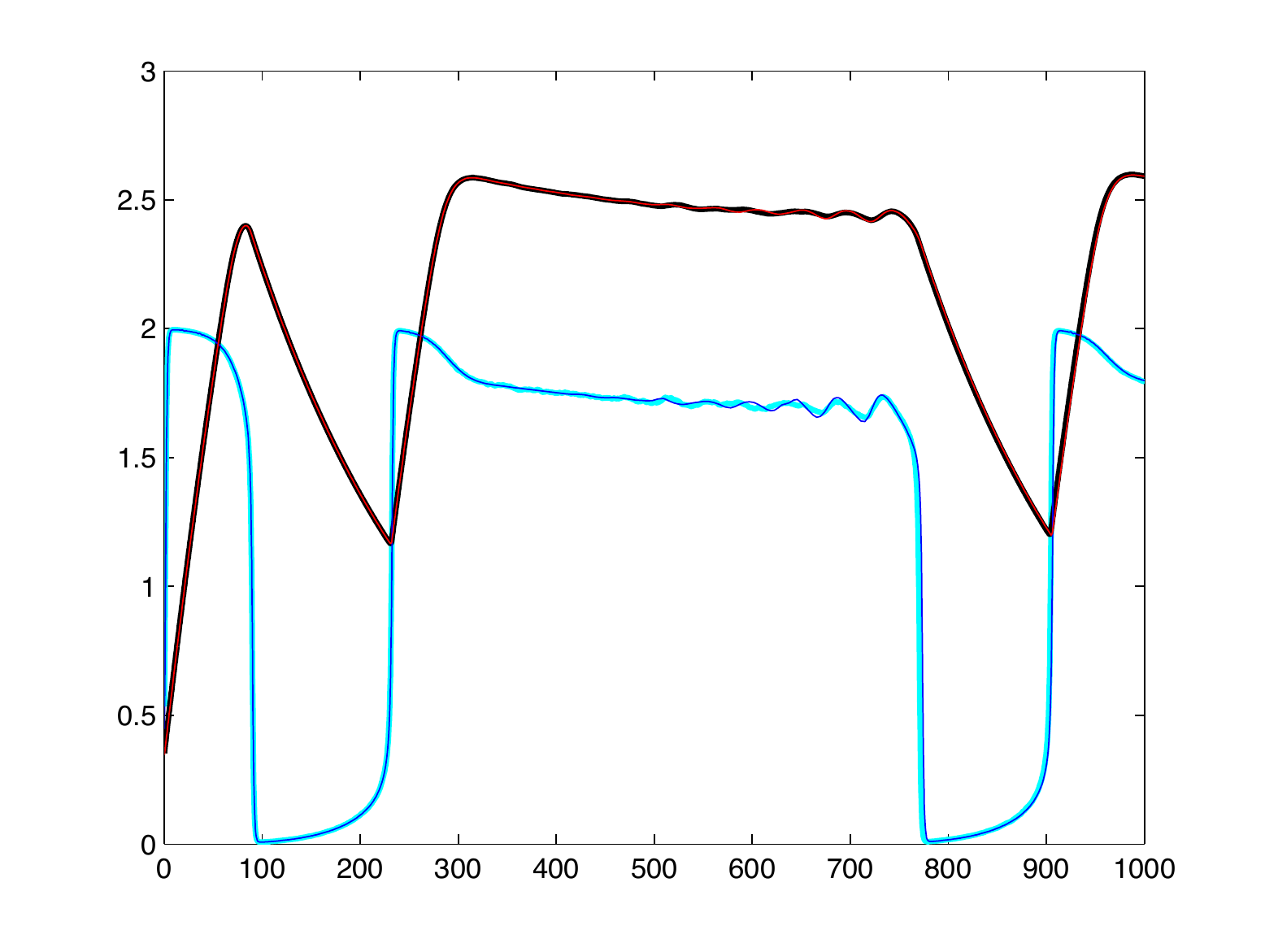}}
	\caption{Same quantities as in Fig.~\ref{fig:MMOsNet} for $N=10\,000$ neurons, (a) same parameters as in Fig.~\ref{fig:drs}(d) and (b): same parameters as in Fig.~\ref{fig:drs} with $k=-1.1$, trajectory presenting a number of visible small oscillations.}
	\label{fig:MMOPrecise}
\end{figure}

\section{Conclusion}
In a number of applied domains such as physics, chemistry and neurosciences, the emergence of complex oscillations such as mixed-mode oscillations have been theoretically related to canard explosions and multiple timescales dynamics of deterministic dynamical systems (see~\cite{md-jg-bk-ck-ho-mw_12} for a review). But the mechanisms underlying the existence of such orbits generally involve a number of interacting agents, for instance neurons, that are characterized by a stochastic activity. It was hence important in that context to show that such orbits persist in these contexts, and to identify the mechanisms accounting for the emergence of such collective behaviors. At the level of one single cell with stochastic perturbations, studies showed that systems presenting canard explosions displayed random MMOs, with the effect of randomly modifying the period of the oscillations and changing the qualitative nature of the solution even in small noise~\cite{berglund-landon:12}. In this paper, the authors analyze stochastic perturbations of Fitzhugh-Nagumo model in the relaxation oscillations regime concluded indeed noise resulted in sensible modifications of the trajectories with random interspike intervals and trajectories similar to MMOs with a random number of small oscillations that asymptotically distribute as a geometric random variable. As of mixed-mode oscillations, Berglund and collaborators~\cite{berglund-gentz-kuehn:12} analyzed the effect of Gaussian noise on a 3 dimensional system with different timescales known to feature mixed-mode oscillations and show that noise changes the global dynamics and mixed mode patterns, and that there exists a maximal amount of noise above which small oscillations are hidden by noise. These studies pointed towards the fact that other phenomena, more resilient to noise, took place in large noisy systems giving rise to such complex oscillations in the presence of multiple timescales. 

Here, we revisited these questions and proposed that self-averaging of the trajectories in large networks could provide a simple generic explanation to the emergence of complex oscillations in large-scale noisy networks. We based our analysis on a classical model of neuronal network, the Wilson-Cowan system, which describes heuristically the dynamics of firing-rates of neurons. This system presented two advantages in our view. First an unsolved mathematical problem was to rigorously derive asymptotic limits of these networks with nonlinear mean-field interactions as the number of neurons tends to infinity. The nature of is indeed very different to usual network equations analyzed in the domain of mean-field limits and propagation of chaos as developed in gas dynamics and extended to neurosciences~\cite{sznitman:89,touboulNeuralfields:11}, since the interacting term is a nonlinear function of a mean-field term, and moreover that the neurons are not governed by a stochastic differential equation but rather by a differential equation with stochastic coefficients $\xi^i_t$. This lead us to extend the theory of mean-field limits\footnote{The proof uses similar tools as usual analysis.} towards this new kind of equations, which we developed in a relatively general fashion in section~\ref{sec:MFLimits}. Applied to Wilson-Cowan equations, this result allowed us to exactly reduce the network equations to a multiple timescales ODE acting on the mean of the solution. Interestingly, these equations reduce to a modified Wilson-Cowan system where the noise level is embedded in the nonlinearities. 

The analysis of the mean equation developed in section~\ref{sec:Bifurcations} had at least two interests. First, the fact that the asymptotic equation on the mean of the solution is itself solution of a Wilson-Cowan type of equation allowed to go deeper in the understanding of these models. And indeed, the behavior of Wilson-Cowan systems in the presence of multiple timescales was still an open problem important to solve since the original Wilson-Cowan model introduced in 1972 in~\cite{wilson-cowan:72} had an intrinsic slow-fast structure that was never explored as such in the literature. Second, this analysis allowed to understand the qualitative role of noise in the dynamics. As was done in~\cite{touboul-hermann:11,touboulNeuralFieldsDynamics:11}, we demonstrated here again that the noise did not have a pure disturbing effect on the trajectories but rather a dramatic qualitative effect on the form of the solutions, governing transitions from stationary to periodic behaviors and complex oscillatory behaviors. 

Of course, the same study would hold in the case of the simpler Amari-Hopfield model analyzed in~\cite{touboul-hermann:11}, where the interaction between neurons is the sum of the firing-rate seen as a sigmoidal transform of the voltage. These models do not take into account the nonlinearities in the synaptic integration taken into account in Wilson-Cowan systems. For such networks, the mean-field equations has Gaussian asymptotic solutions, whose mean and standard deviation satisfy a closed set of ordinary differential equations~\cite{touboul-hermann:11}. Our analysis can easily be extended to such systems, where the limit equations obtained can be shown to have Gaussian solution whose mean satisfy the set of ODEs:
\[
	\tau_{\alpha}\der{\mu_{\alpha}}{t}=-\mu_{\alpha}+\sum_{\beta=1}^P J_{\alpha\beta}G_{\alpha\beta}(\mu_{\beta},\sigma_{\alpha})
\]
where $G_{\alpha\beta}$ is similar to our effective sigmoid function~\ref{eq:EffectiveSigmoid} and embeds dependence on the noise level $\sigma_{\alpha}$. In the presence of multiple timescales, we expect to observe the same kind of phenomena, and a similar dependence on noise levels, and these phenomena may be rigorously demonstrated in the same flavor as we did in the present paper. 

Such rigorous results as proved in the present manuscript provide a proof of concept of the role of noise in large-scale multiple timescales systems. The question that may arise is whether similar phenomena arise in networks of excitable elements. In that case, the rigorous reduction to ODEs is no more possible, and proving rigorous results amounts studying infinite-dimensional systems given by non-local partial differential equations (see e.g.~\cite{touboulNeuralfields:11} for a formulation of the problem in excitable systems). However, preliminary numerical explorations of the problem tend to show that the same type of phenomena arise in these systems. The deep analysis of these systems and the existence of canards and complex oscillations in relationship with noise levels is a direct perspective of this work we are currently analyzing. 

Another deep problem raised by the present analysis is the systematic characterization of the finite-size effects in bistable regimes. Indeed, we have observed that when there is a unique attractor of the mean-field equations, the finite-sized networks approximated fairly the predicted behavior, whereas in the bistable regime, switches occur randomly between the two attractors, but with a surprising regularity. Using the theory of the fluctuations around the mean-field limits~\cite{sznitman:84a,fernandez1997hilbertian}, we will explore in depth the origin of this regularity in a forthcoming study. 

The results obtained in this paper provide further insight on the role of noise in the dynamics of neuronal networks. More than a pure perturbation of the trajectories, noise shows to govern the dynamics of the networks, and moreover, the presence of canard explosions further assess that this dependence is extremely sensitive. This shows that stochastic networks displaying complex oscillations may be related to very finely tuned noise levels. This fine characterization of the noise may lead to unravel precisely noise levels in large-scale neuronal networks.

\end{document}